\documentclass[11pt]{article}

\usepackage{amsmath,amssymb,amsfonts}
\usepackage{color,xcolor}
\usepackage{mathrsfs}
\usepackage[linktocpage,colorlinks,linkcolor=blue,anchorcolor=blue, citecolor=blue,urlcolor=blue]{hyperref}

\newtheorem{theorem}{Theorem}[section]

\newtheorem{definition}[theorem]{Definition}
\newtheorem{lemma}[theorem]{Lemma}
\newtheorem{proposition}[theorem]{Proposition}
\newtheorem{example}[theorem]{Example}
\newtheorem{algorithm}[theorem]{Algorithm}

\newcommand{\Om}{{\bf\Omega}}
\newcommand{\BS}{{\bf S}}
\newcommand{\ii}{\textnormal{\bf i}}
\newcommand{\re}{\textnormal{Re}\,}
\newcommand{\im}{\textnormal{Im}\,}
\newcommand{\C}{\mathbb{C}}

\newcommand{\R}{\mathbb{R}}

\newcommand{\F}{\mathcal{F}}
\newcommand{\G}{\mathcal{G}}
\newcommand{\HI}{\mathcal{H}}
\newcommand{\X}{\mathcal{X}}
\newcommand{\ba}{\boldsymbol{a}}

\newcommand{\bx}{\boldsymbol{x}}
\newcommand{\by}{\boldsymbol{y}}
\newcommand{\bz}{\boldsymbol{z}}

\newcommand{\bu}{\boldsymbol{u}}
\newcommand{\bv}{\boldsymbol{v}}
\newcommand{\bxi}{\boldsymbol{\xi}}
\newcommand{\ov}{\overline}

\newcommand{\T}{\textnormal{T}}
\newcommand{\st}{\textnormal{s.t.}}
\newcommand{\ex}{\textnormal{E}\,}

\newcommand{\conv}{\textnormal{conv}\,}
\newcommand{\Prob}{\textnormal{Prob}\,}

\begin{document}

\title{Approximation Algorithms for Optimization of Real-Valued General Conjugate Complex Forms}

\author{
Taoran FU
\thanks{School of Mathematical Sciences, Shanghai Jiao Tong University, Shanghai 200240, China. Email: \mbox{taoran30@sjtu.edu.cn}.}
    \and
Bo JIANG
\thanks{Research Center for Management Science and Data Analytics, School of Information Management and Engineering, Shanghai University of Finance and Economics, Shanghai 200433, China. Email: \mbox{isyebojiang@gmail.com}.}
    \and
Zhening LI
\thanks{Department of Mathematics, University of Portsmouth, Portsmouth PO1 3HF, United Kingdom. Email: \mbox{zheningli@gmail.com}.}
}

\date{\today}

\maketitle

\begin{abstract}

Complex polynomial optimization has recently gained more and more attention in both theory and practice. In this paper, we study the optimization of a real-valued general conjugate complex form over various popular constraint sets including the $m$-th roots of complex unity, the complex unit circle, and the complex unit sphere. A real-valued general conjugate complex form is a homogenous polynomial function of complex variables as well as their conjugates, and always takes real values. General conjugate form optimization is a wide class of complex polynomial optimization models, which include many homogenous polynomial optimization in the real domain with either discrete or continuous variables, and Hermitian quadratic form optimization as well as its higher degree extensions. All the problems under consideration are NP-hard in general and we focus on polynomial-time approximation algorithms with worst-case performance ratios. These approximation ratios improve previous results when restricting our problems to some special classes of complex polynomial optimization, and improve or equate previous results when restricting our problems to some special classes of polynomial optimization in the real domain. These algorithms are based on tensor relaxation and random sampling. Our novel technical contributions are to establish the first set of probability lower bounds for random sampling over the $m$-th root of unity, the complex unit circle, and the complex unit sphere, and propose the first polarization formula linking general conjugate forms and complex multilinear forms.

\vspace{0.25cm}


\vspace{0.25cm}


\end{abstract}

\section{Introduction}\label{sec:introduction}

Polynomial optimization has been one of the main research topics in mathematical optimization in the past decade due to its sophisticated theory in semi-algebraic geometry~\cite{L09}, its algorithmic effects in both continuous optimization and discrete optimization~\cite{AL12}, as well as its enormous applications such as biomedical engineering, control theory, graph theory, investment science, material science, quantum mechanics, signal processing, speech recognition~\cite{LHZ12}. Most research emphasis has been put on polynomial optimization in the domain of real numbers. In recent years, motivated by a large number of applications, especially in signal processing, various types of polynomial optimization models in the complex domain were proposed and studied. Aittomaki and Koivunen~\cite{AK09} formulated the beampattern optimization problem as a complex multivariate quartic minimization problem. Chen et al.~\cite{CV09} considered the joint optimization problem of waveforms and receiving filters in multiple-input and multiple-output radar, and relaxed it to a quartic complex polynomial optimization model. Hilling and Sudbery~\cite{HS10} constructed a complex polynomial optimization with the spherical constraint in the area of quantum entanglement. Aubry et al.~\cite{AMJZ13} introduced a cognitive approach to design a special class of waveforms by optimizing a complex quartic polynomial with a constant modulus constraint. Very recently, the application of complex polynomial optimization to electricity transmission networks was discovered and investigated by Josz~\cite{J16}, and its application in power system state estimation was studied by Madani et al.~\cite{MLB16}.

On the algorithmic aspect, the traditional sum-of-squares method by Lasserre~\cite{L01} for general polynomial optimization problems has been extended to complex polynomial optimizations; see e.g.,~\cite{DW12,JM15}. Since polynomial optimization problems are generally NP-hard, various polynomial-time approximation algorithms have been proposed for solving certain classes of high-degree polynomial optimization models---a summary of research can be found in the monograph of Li et al.~\cite{LHZ12}. Improvements on approximation ratios of these polynomial optimization models have been recently made by He et al.~\cite{HJLZ14} and Hou and So~\cite{HS14}. In the context of complex polynomial optimization, approximation algorithms are mostly proposed for the quadratic models. Complex quadratic form optimization under the $m$-th roots of unity constraints and the complex unit circle constraints have been studied in~\cite{SZY07,ZH06}. Huang and Zhang~\cite{HZ10} also discussed bilinear complex polynomial optimization models. Beyond quadratics, Jiang et al.~\cite{JLZ14} studied approximation algorithms for various high-degree complex polynomial optimization under the $m$-th roots of unity constraints, the complex unit circle constraints, and complex spherical constraints. 

In almost all the complex optimization models mentioned above, the objective function to be optimized is the real part of a complex polynomial function rather than the function itself since it is not real-valued. Very recently, Jiang et al.~\cite{JLZ16} provided a necessary and sufficient condition under which complex polynomials always take real values. Based on this condition, they proposed a very wide class of real-valued complex polynomial functions, called {\em general conjugate forms}, which include all the complex objective functions studied in~\cite{SZY07,ZH06,HZ10,JLZ14} as special cases, as well as all homogeneous polynomial functions in the real domain. In this paper, we are primarily interested in the real-valued general conjugate form optimization under various popular constraints in complex variables, such as the $m$-th roots of unity, the complex unit circle, and the complex spherical constraints. The emphasis is to propose polynomial-time approximation algorithms and analyze their performances. Originated from previous researches in probability estimation of random sampling~\cite{KN08,HJLZ14}, tensor relaxation and polarization formula~\cite{HLZ10} and feasible solution reconstruction~\cite{JLZ14}, we develop some new techniques and results on that line, which enable us to study a new and much more general class of complex polynomial optimization models that covers and improves many existing researches in the literature. The main contributions of the paper are as follows:
\begin{itemize}
  \itemsep0pt
  \item We propose the first polarization formula relating general conjugate forms and complex multilinear forms;
  \item We study random sampling over the $m$-th root of unity, the complex unit circle and the complex sphere, and provide some first probability lower bounds;
  \item We propose new approximation algorithms for complex multilinear form optimization over various complex constraints, whose approximation ratios improve and generalize that of~\cite{HZ10,JLZ14};
  \item We propose some first approximation algorithms for real-valued general conjugate form optimization over various complex constraints, whose approximation ratios improve that of~\cite{ZH06,SZY07,JLZ14} when restricting to some special classes of complex polynomial optimization, and improve or equate that of~\cite{HLZ10,S11,ZQY12,ZCTW12,HLZ13,HJLZ14} when restricting to some special classes of polynomial optimization in the real domain.
\end{itemize}

This paper is organized as follows. We start with preparations of various notations, definitions of various complex functions and complex optimization models in Section~\ref{sec:prepare}. In Section~\ref{sec:link}, we present a polarization formula that links general conjugate forms and complex multilinear forms, paving a way to study general conjugate form optimization via complex multilinear form optimization. In Section~\ref{sec:key}, we discuss some key probability bounds for random sampling over the $m$-th roots of unity, the complex unit circle and the complex sphere, a fundamental step in deriving improved approximation algorithms for complex multilinear form optimization. Polynomial-time approximation algorithms with improved approximation ratios for complex multilinear form optimization over various types of constraint sets are proposed and analyzed in Section~\ref{sec:mulform}. Finally, by applying the linkage between general conjugate forms and complex multilinear forms, approximation algorithms with guaranteed worst-case performance ratios for general conjugate form optimization under various constraints are discussed in Section~\ref{sec:gform}.

\section{Preparations}\label{sec:prepare}

Throughout this paper we use usual lowercase letters, boldface lowercase letters, capital letters, and calligraphic letters to denote scalars, vectors, matrices, and tensors, respectively. For example, a scalar $x$, a vector $\bx$, a matrix $X$, and a tensor $\X$. We use subscripts to denote their components, e.g., $x_i$ being the $i$-th entry of a vector $\bx$, $X_{ij}$ being the $(i,j)$-th entry of a matrix $X$, and $\X_{ijk}$ being the $(i,j,k)$-th entry of a third order tensor $\X$. As usual, the field of real numbers and the field of complex numbers are denoted by $\R$ and $\C$, respectively.

For any complex number $z=a+\ii b\in\C$ with $a,b\in\R$, its real part and imaginary part are denoted by $\re z:=a$ and $\im z:=b$, respectively. Its modulus is denoted by $|z|:=\sqrt{\ov{z}z}=\sqrt{a^2+b^2}$, where $\ov{z}:=a-\ii b$ denotes the conjugate of $z$. 
The $L_p$-norm ($1\le p\le \infty$) of a complex vector $\bx\in\C^n$ is defined as $\|\bx\|_p:=\left(\sum_{i=1}^n |x_i|^p\right)^{\frac{1}{p}}$.

\subsection{Complex multilinear forms and homogenous complex polynomials}

Given a $d$-th order complex tensor $\F=(\F_{i_1i_2\dots i_d})\in\C^{n_1\times n_2\times \dots \times n_d}$, its associated complex multilinear form $F$ is defined as
$$
  F(\bx^1,\bx^2,\dots,\bx^d):= \sum_{i_1=1}^{n_1}\sum_{i_2=1}^{n_2}\dots\sum_{i_d=1}^{n_d} \F_{i_1i_2\dots i_d}\,x^1_{i_1}x^2_{i_2}\dots x^d_{i_d},
$$
where the variables $\bx^k\in\C^{n_k}$ for $k=1,2,\dots,d$. Closely related to a multilinear form is a homogeneous complex polynomial function, or explicitly
$$
  f(\bx):= \sum_{1\le i_1 \le i_2 \le \dots \le i_d\le n} a_{i_1i_2\dots i_d}x_{i_1}x_{i_2}\dots x_{i_d},
$$
where the variable $\bx\in\C^n$. Associated with any homogeneous complex polynomial is a symmetric complex tensor $\F\in\C^{n^d}$, i.e., its entries $\F_{i_1i_2\dots i_d}$'s are invariant under permutations of their indices $\{i_1,i_2,\dots,i_d\}$. In this sense,
$$
\F_{i_1i_2\dots i_d}=\frac{a_{i_1i_2\dots i_d}}{|\Pi(i_1i_2\dots i_d)|} \quad\forall\, 1\le i_1\le i_2\le\dots\le i_d\le n,
$$
where $\Pi(i_1i_2\dots i_d)$ is the set of all distinct permutations of the indices $\{i_1,i_2,\dots, i_d\}$. In light of a multilinear form $F$ associated with a symmetric tensor $\F$, homogeneous polynomial $f(\bx)$ is obtained by letting $\bx^1=\bx^2=\dots=\bx^d=\bx$, i.e., $f(\bx)=F(\underbrace{\bx,\bx,\dots,\bx}_d)$. We call such an $\F$ to be the tensor representation of the homogeneous complex polynomial $f(\bx)$.

\subsection{General conjugate forms and their tensor representations}

A multivariate {\em conjugate} complex polynomial $c(\bx)$ is a polynomial function of variables $\bx,\ov{\bx}\in\C^n$. Without conjugate terms, real-valued complex polynomial functions are meaningless as otherwise they become constant functions. Restricting to homogeneous ones, Jiang et al.~\cite{JLZ16} proposed general conjugate forms.
\begin{definition}[General conjugate form~\cite{JLZ16}] \label{def:gform}
A general conjugate form of the variable $\bx\in\C^n$ is defined as
\begin{equation}\label{eq:gform}
g(\bx)=\sum_{k=0}^d \, \sum_{1\le i_1\le i_2\le \dots \le i_k \le n}\, \sum_{1\le j_1 \le j_2\le \dots \le j_{d-k} \le n} a_{i_1i_2\dots i_k,j_1j_2\dots j_{d-k}}\ov{x_{i_1}x_{i_2}\dots x_{i_k}}x_{j_1} x_{j_2}\dots x_{j_{d-k}}.
\end{equation}
\end{definition}
Essentially, it is the summation of all the possible $d$-th degree monomials, allowing any number of conjugate variables as well as usual variables in each monomial. This, however, does not require the number of conjugate variables being the same as the number of usual variables in any monomial, a special type of general conjugate forms called symmetric conjugate forms defined in~\cite{JLZ16}. Jiang et al.~\cite{JLZ16} proved that a general conjugate form taking real values for all $\bx\in\C^n$ if and only if the coefficients of each pair of conjugate monomials are conjugate to each other, i.e., $a_{i_1i_2\dots i_k,j_1j_2\dots j_{d-k}}=\ov{a_{j_1j_2\dots j_{d-k},i_1i_2\dots i_k}}$ in~\eqref{eq:gform}. It worth mentioning that restricting to the quadratic case, real-valued general conjugate forms include Hermitian quadratic forms (which are also real-valued) as a subclass since the latter one requires exact one conjugate variable and one usual variable in any monomial.

The tensor representation for a real-valued general conjugate form is interesting, which is explicitly characterized as follows.
\begin{definition}[Conjugate super-symmetric tensor~\cite{JLZ16}] \label{def:css}
An even dimensional tensor $\G\in\C^{(2n)^d}$ is called conjugate super-symmetric if \\
(i) $\G$ is symmetric, i.e., $\G_{i_1i_2\dots i_d} = \G_{j_1j_2\dots j_d}$ for all $(j_1j_2\dots j_d) \in \Pi(i_1i_2\dots i_d)$,
and \\
(ii) $\G_{i_1i_2\dots i_d} = \ov{\G_{j_1j_2\dots j_d }}$ holds for all $1\le i_1, i_2, \dots, i_d, j_1, j_2, \dots, j_d\le 2n$ with $|i_k -j_k| = n$ for $k=1,2,\dots,d$.
\end{definition}
There is one-to-one correspondence between $n$-dimensional $d$-th degree real-valued general conjugate forms and $2n$-dimensional $d$-th order conjugate super-symmetric tensors~\cite{JLZ16}. In particular, for any conjugate super-symmetric $\G\in\C^{(2n)^d}$, the corresponding real-valued general conjugate form can be obtained by
\begin{equation}\label{eq:gform2}
  g(\bx)=G\bigg(\underbrace{\binom{\ov\bx}{\bx},\binom{\ov\bx}{\bx},\dots,\binom{\ov\bx}{\bx}}_d\bigg).
\end{equation}
A simple example of a complex quadratic polynomial (matrix case) is shown below.
\begin{example}
  Given a conjugate super-symmetric second order tensor (matrix) $G=\left( \begin{smallmatrix} \ii & 0 & 1 & 2\\ 0 & 0 & 2 & 0\\ 1 & 2 & -\ii & 0 \\ 2 & 0 &0 & 0 \end{smallmatrix}  \right)\in\C^{4^2}$,  the corresponding general conjugate form is $$g(\bx)=(\ov{x_1},\ov{x_2},x_1,x_2)G(\ov{x_1},\ov{x_2},x_1,x_2)^{\T}=\ii{\ov{x_1}}^2 + 2\ov{x_1}x_1 + 4\ov{x_1}x_2 + 4 \ov{x_2}x_1 - \ii {x_1}^2,$$
  which always takes real values for any $x_1,x_2\in\C$.
\end{example}

\subsection{Complex constraint sets}\label{sec:constraint}

The following commonly encountered constraint sets for complex polynomial optimization are considered in this paper:
\begin{itemize}
  \itemsep0pt
  \item The $m$-th roots of unity: $\Om_m=\left\{1,\omega_m,\dots,\omega_m^{m-1}\right\}$, where $\omega_m = e^{\ii\frac{2\pi}{m}}= \cos{\frac{2\pi}{m}} + \ii \sin {\frac{2\pi}{m}}$. Denote $\Om_m^n=\left\{\bx\in\C^n:x_i\in\Om_m,\,i=1,2,\dots,n\right\}$.
  \item The complex unit circle: $\Om_\infty=\{z \in \C:|z|=1 \}$. Denote $\Om_\infty^n=\{\bx\in\C^n:x_i\in\Om_\infty,\, i=1,2,\dots,n \}$.
  \item The complex sphere: $\BS^n = \left\{\bx\in\C^n\ : \|\bx\|_2=1\right\}.$
\end{itemize}
Throughout this paper, we assume $m\ge3$, to ensure that the decision variables being considered are essentially complex.

\subsection{Complex polynomial optimization models}\label{sec:model}

This main purpose of this paper is to study approximation algorithms for real-valued general conjugate form optimization over three types of constraint sets mentioned in Section~\ref{sec:constraint}. Specifically, given a real-valued general conjugate form $g(\bx)$ associated with a conjugate super-symmetric tensor $\G$, we study the following optimization models,
$$
\begin{array}{lll}
(G_m) & \max & g(\bx) \\
             & \st & \bx \in \Om_m^n; \\
(G_\infty) & \max  & g(\bx) \\
           & \st & \bx \in \Om_\infty^n; \\
(G_S) & \max      & g(\bx) \\
      & \st & \bx \in \BS^n.
\end{array}
$$

As $g(\bx)=G\bigg(\underbrace{\binom{\ov\bx}{\bx},\binom{\ov\bx}{\bx},\dots,\binom{\ov\bx}{\bx}}_d\bigg)$ where $\bx\in\C^n$, the tensor relaxation approach~\cite{HLZ10} is applied to study these models, i.e., relaxing the objective function $g(\bx)$ to $G(\bx^1,\bx^2,\dots,\bx^d)$ where $\bx^k\in\C^{2n}$ for $k=1,2,\dots,d$. Therefore, we first study the following optimization models,
$$
\begin{array}{lll}
(L_m) & \max & \re F(\bx^1,\bx^2,\dots,\bx^d) \\
             & \st & \bx^k \in \Om_m^{n_k} ,\, k=1,2,\dots,d; \\
(L_\infty) & \max  & \re F(\bx^1,\bx^2,\dots,\bx^d) \\
           & \st & \bx^k \in \Om_\infty^{n_k} ,\, k=1,2,\dots,d; \\
(L_S) & \max      & \re F(\bx^1,\bx^2,\dots,\bx^d) \\
      & \st & \bx^k \in \BS^{n_k} ,\, k=1,2,\dots,d.
\end{array}
$$
where $F$ is a complex multilinear form associated with a complex tensor $\F\in\C^{n_1\times n_2\times \dots \times n_d}$. The real-part operator has to be put in the objective function as a complex multilinear form cannot always take real values.

\subsection{Polynomial-time approximation algorithms}

For any maximization problem $(P):\max_{\bx\in {\bf X}}p(\bx)$ studied in this paper, we denote $v_{\max}(P)$ to be the optimal value and $v_{\min}(P)$ to be the optimal value of its minimization counterpart $\min_{\bx\in {\bf X}}p(\bx)$.
\begin{definition}
(i) A maximization problem $(P):\max_{\bx\in {\bf X}}p(\bx)$ admits a polynomial-time approximation algorithm with approximation ratio $\rho\in(0,1]$, if $v_{\max}(P)\ge0$ and a feasible solution $\bz\in {\bf X}$ can be found in polynomial time, such that $p(\bz)\ge \rho\, v_{\max}(P)$.\\
(ii) A maximization problem $(P):\max_{\bx\in {\bf X}}p(\bx)$ admits a polynomial-time approximation algorithm with {\em relative} approximation ratio $\rho\in(0,1]$, if a feasible solution $\bz\in {\bf X}$ can be found in polynomial time, such that $p(\bz)-v_{\min}(P)\ge \rho\left(v_{\max}(P)-v_{\min}(P)\right)$.
\end{definition}
There is no evidence that one type of approximation ratios is better than or implies the other. Whether a usual approximation ratio is obtainable or it has to be a relative approximation ratio is really depend on the nature of the optimization model. For some problems, such as the objective function is always negative implying that $v_{\max}(P)) \le 0$, only a {\em relative} approximation ratio can be obtained.

All the optimization models considered in this paper are NP-hard in general, even restricted to the real domain. Polynomial-time randomized algorithms with worst-case approximation ratios are proposed for these models, when the degree of these complex polynomial functions, $d$, is fixed. These approximation ratios depend only on the dimensions of the problems, or in other words, they are data-independent.

\section{Polarization identity of general conjugate forms}\label{sec:link}

As mentioned in Section~\ref{sec:model}, complex multilinear form relaxations are applied to study real-valued general conjugate form optimization models. This section is devoted to establishing an identity linking these two complex polynomial functions. In the literature, such identities are usually called polarization identities. He et al.~\cite{HLZ10} first established a polarization identity linking multilinear forms to homogenous polynomials. So~\cite{S11} proposed a polarization identity for multiquadratic forms and He et al.~\cite{HLZ13} further extended such identity to mixed forms. These identities can be applied to both the real and the complex domains. In the complex domain specifically, Jiang et al.~\cite{JLZ14} established a polarization identity liking complex multilinear forms to symmetric conjugate forms, a special class of general conjugate forms. Our main result in this section is as follows.
\begin{theorem}\label{thm:link}
Let $m\ge3$ be an integer or $m=\infty$. Suppose $g(\bx)$ with $\bx\in\C^n$ is a real-valued general conjugate form associated with a conjugate super-symmetric tensor $\G\in\C^{(2n)^d}$. If $\xi_1,\xi_2,\dots, \xi_d$ are i.i.d.\ uniformly on $\Om_m$, then for any $\bx^1,\bx^2,\dots,\bx^d, \by^1,\by^2,\dots,\by^d \in \C^{n}$
$$
\ex\left[ \left(\prod_{i=1}^d \ov{\xi_i}\right)g \left( \sum_{k=1}^d\left(\ov{\xi_k\bx^k}+\xi_k\by^k\right) \right) \right]
=d!\, G\left(\binom{\bx^1}{\by^1},\binom{\bx^2}{\by^2},\dots,\binom{\bx^d}{\by^d}\right).
$$
\end{theorem}
\begin{proof}
According to~\eqref{eq:gform2}, we have
\begin{align*}
&~~~~\ex\left[\left(\prod_{i=1}^d \ov{\xi_i}\right) g \left( \sum_{k=1}^d\left(\ov{\xi_k\bx^k}+\xi_k\by^k\right) \right) \right]\\
&=\ex\left[\left(\prod_{i=1}^d \ov{\xi_i}\right) G\left(
\binom {\sum_{k=1}^d(\xi_k\bx^k+\ov{\xi_k\by^k})} {\sum_{k=1}^d(\ov{\xi_k\bx^k}+\xi_k\by^k)},
\binom {\sum_{k=1}^d(\xi_k\bx^k+\ov{\xi_k\by^k})} {\sum_{k=1}^d(\ov{\xi_k\bx^k}+\xi_k\by^k)}, \dots,
\binom {\sum_{k=1}^d(\xi_k\bx^k+\ov{\xi_k\by^k})} {\sum_{k=1}^d(\ov{\xi_k\bx^k}+\xi_k\by^k)} \right)\right]\\
&=\ex\left[\left(\prod_{i=1}^d \ov{\xi_i}\right) G\left(
\sum_{k=1}^d \left(\xi_k\binom{\bx^k}{\by^k}+\ov{\xi_k}\binom{\ov{\by^k}}{\ov{\bx^k}}\right),
\sum_{k=1}^d \left(\xi_k\binom{\bx^k}{\by^k}+\ov{\xi_k}\binom{\ov{\by^k}}{\ov{\bx^k}}\right), \dots,
\sum_{k=1}^d \left(\xi_k\binom{\bx^k}{\by^k}+\ov{\xi_k}\binom{\ov{\by^k}}{\ov{\bx^k}}\right)  \right) \right] \\
&=\ex\left[\left(\prod_{i=1}^d \ov{\xi_i}\right) G\left(\sum_{k=1}^{2d}\eta_k\bz^k,\sum_{k=1}^{2d}\eta_k\bz^k, \dots, \sum_{k=1}^{2d}\eta_k\bz^k\right)\right]\\
&=\ex\left[\sum_{k_1=1}^{2d} \sum_{k_2=1}^{2d} \dots \sum_{k_d=1}^{2d} \left(\prod_{i=1}^d \ov{\xi_i}\right) \left(\prod_{j=1}^d \eta_{k_{j}}\right) G (\bz^{k_1},\bz^{k_2},\dots ,\bz^{k_d})\right],
\end{align*}
where the last equality is due to the multilinearity of $G$,
$$\bz^k:=\binom{\bx^k}{\by^k} \mbox{ for } k=1,2,\dots, d \mbox{ and } \bz^k:=\binom{\ov{\by^{k-d}}}{\ov{\bx^{k-d}}} \mbox{ for } k=d+1,d+2,\dots,2d,$$
and
$$\eta_k=\xi_k \mbox{ for } k=1,2,\dots, d \mbox{ and } \eta_k=\ov{\xi_{k-d}} \mbox{ for } k=d+1,d+2,\dots,2d.$$

Let us take a close look at $\ex\left[\left(\prod_{i=1}^d \ov{\xi_i}\right) \left(\prod_{j=1}^d \eta_{k_{j}}\right)\right]$ for all the possible $k_j$'s. Since $m\ge 3$ or $m=\infty$, it is obvious that
\begin{equation}\label{first-second-moments}
\ex\xi_i=0,~\ex\xi_i^2=0,~\ex\ov{\xi_i}^2=0,\mbox{ and }\ov{\xi_i}\xi_i=1\mbox{ for }i=1,2,\dots,d.
\end{equation}
Consider $d$ sets of index couples $\{k,d+k\}$ for $k=1,2,\dots,d$, and we discuss the distribution of $\{k_1,k_2, \dots ,k_d\}$ in these index couples via three cases.

(i) None of the two $k_j$'s belongs to the same index couple and $\max_{1\le j\le d}\{k_j\}\le d$. In this case, $(k_1,k_2, \dots ,k_d) \in \Pi(1,2,\dots ,d)$, i.e., a permutation of $\{1,2,\dots,d\}$. We have
\begin{equation*}
\ex\left[\left(\prod_{i=1}^d \ov{\xi_i}\right) \left(\prod_{j=1}^d \eta_{k_{j}}\right)\right]
=\ex \left[\left(\prod_{i=1}^d \ov{\xi_i}\right) \left(\prod_{j=1}^d \eta_j\right)\right]
=\ex \left[\left(\prod_{i=1}^d \ov{\xi_i}\right) \left(\prod_{j=1}^d \xi_j\right)\right]
=\prod_{i=1}^d \ex[\ov{\xi_i}\xi_i]=1.
\end{equation*}

(ii) None of the two $k_j$'s belongs to the same index couple and $\max_{1\le j\le d}\{k_j\}> d$. If we pick any $\ell$ with $k_\ell>d$, then $\eta_{k_\ell}=\ov{\xi_{k_\ell-d}}$ and none of the other $k_j$'s belongs to $\{k_\ell-d,k_\ell\}$. We have
\begin{equation*}
\ex\left[\left(\prod_{i=1}^d \ov{\xi_i}\right) \left(\prod_{j=1}^d \eta_{k_{j}}\right)\right]
=\ex\ov{\xi_{k_\ell-d}}^2 \,\ex\left[\left(\prod_{1\le i\le d,\,i\ne k_\ell-d}\ov{\xi_i}\right) \left(\prod_{1\le j\le d,\,j\neq \ell} \eta_{k_j}\right)\right]
=0.
\end{equation*}

(iii) There exists an $\ell~(1\le \ell\le d)$ such that none of $k_j$'s belongs to $\{\ell,\ell+d\}$. We have
\begin{equation*}
\ex\left[\left(\prod_{i=1}^d \ov{\xi_i}\right) \left(\prod_{j=1}^d \eta_{k_{j}}\right)\right] =\ex\ov{\xi_\ell} \,\ex\left[\left(\prod_{1\le i\le d,\,i\ne \ell}\ov{\xi_i}\right)\left(\prod_{j=1}^d \eta_{k_{j}}\right)\right]=0.
\end{equation*}

Therefore, $\ex\left[\left(\prod_{i=1}^d \ov{\xi_i}\right) \left(\prod_{j=1}^d \eta_{k_{j}}\right)\right]\neq 0$ if and only if $(k_1,k_2, \dots ,k_d) \in \Pi(1,2,\dots ,d)$. As the number of different permutations in $\Pi(1,2,\dots ,d)$ is $d!$, by taking into account of the symmetricity of $\G$, it follows that
\begin{align*}
\ex\left[\sum_{k_1=1}^{2d} \sum_{k_2=1}^{2d} \dots \sum_{k_d=1}^{2d} \left(\prod_{i=1}^d \ov{\xi_i}\right) \left(\prod_{j=1}^d \eta_{k_{j}}\right) G (\bz^{k_1},\bz^{k_2},\dots ,\bz^{k_d})\right]
&=d!\, G(\bz^1,\bz^2,\dots ,\bz^d)\\
&=d!\, G\left(\binom{\bx^1}{\by^1},\binom{\bx^2}{\by^2},\dots,\binom{\bx^d}{\by^d}\right),
\end{align*}
proving the identity.
\end{proof}

\section{Main probability bounds}\label{sec:key}

This section is devoted to some key probability inequalities that will be used in deriving approximation algorithms in Section~\ref{sec:mulform}. In particular, we consider the inner product between a fixed complex vector and a random complex vector, and establish a nontrivial lower bound on the event that such inner product is larger than a certain threshold. In the real domain, such inequalities are extremely useful in designing randomized approximation algorithms, see e.g.,~\cite{KN08,HJLZ14}. Although the main idea in the proof of the inequality originates from~\cite{KN08}, the case for a random complex vector has a more sophisticated structure leading to more general and useful results.

In the real domain, Khot and Naor~\cite{KN08} proved that for every $\delta\in(0,\frac{1}{2})$, there is a constant $c(\delta)>0$ such that if $\bxi\in\R^n$ whose entries are i.i.d.\ symmetric Bernoulli random variables (taking $\pm 1$ with equal probability), then for any $\ba\in\R^n$,
$$
\Prob\left\{\ba^{\T}\bxi \ge \sqrt{\frac{\delta \ln{n}}{n}}\, \| \ba\|_1 \right\} \ge \frac{c(\delta)}{n^{\delta}}.
$$
In another setting, Brieden et al.~\cite{BGKKLS98} showed that if $\bxi\in\R^n$ is drawn uniformly on the unit sphere $\{\bx\in\R^n:\|\bx\|_2=1\}$, then for any $\ba\in\R^n$,
$$ \Prob\left\{\ba^{\T}\bxi\ge\sqrt{\frac{\ln n}{n}}\|\ba\|_2\right\} \ge \frac{1}{10\sqrt{\ln n}} \left(1-\frac{\ln n}{n}\right) ^{\frac{n-1}{2}}.$$
A refinement of the above result and extensions to polynomial functions were discussed by He et al.~\cite{HJLZ14}. In the complex domain, the symmetric Bernoulli random variable obviously extends to the uniform distribution on $\Om_m$, and the unform distribution on the unit sphere extends to the uniform distribution over the complex unit sphere. Our results are presented in Theorem~\ref{thm:key1} and Theorem~\ref{thm:key2}, respectively. Before deriving these inequalities, let us first review a useful inequality.

\begin{lemma}[Berry-Esseen inequality~\cite{E56,I10}]\label{thm:be}
Let $\eta_1,\eta_2,\dots,\eta_n$ be independent real random variables with $\ex\eta_i=0$, $\ex \eta_i^2=\sigma_i^2>0$ and $\ex |\eta_{i}|^{3}=\kappa_i< \infty$ for $i=1,2,\dots,n$. Denote $s_n=\sum_{i=1}^n\eta_i\Big{/}\sqrt{\sum_{i=1}^n\sigma_i^2}$
to be the normalized $n$-th partial sum and $S_{n}$ to be the cumulative distribution function of $s_{n}$. It follows that
$$
\sup_{x\in \R}|S_{n}(x)-N(x)| \le \frac{c_0\sum_{i=1}^{n}\kappa_i}{\left(\sum_{i=1}^{n}\sigma_{i}^2\right)^{\frac{3}{2}}},
$$
where $c_0\in(0.4097, 0.56)$ is a constant and $N(t):= \int_{-\infty}^t \frac {1}{\sqrt{2\pi }} e^{-\frac{x^2}{2}} dx$ is the cumulative distribution function of the standard normal distribution.
\end{lemma}

The following moments' estimation will be used frequently in this section.
\begin{lemma}\label{thm:complex}
Let $m\ge3$ be an integer or $m=\infty$. Let $\bxi=(\xi_1,\xi_2,\dots,\xi_{n})^{\T}\in\C^n$ whose entries are i.i.d.\ uniformly on $\Om_m$ and $\ba \in \C^{n}$ be fixed. If we define $\eta=\re(\ba^{\T}\bxi)$, then
\begin{align*}
  \ex\eta&=0,\\
  \ex\eta^2&=\frac{1}{2}\sum_{i=1}^n|a_i|^2, \\
  \ex\eta^4&=\left\{
    \begin{array}{ll}
    \frac{1}{16} \sum_{i=1}^{n}(a_{i}^4+\ov{a_{i}}^{4})+\frac{3}{8} \sum_{i=1}^{n}|a_{i}|^{4}+\frac{3}{2} \sum_{1\le i<j \le n}|a_{i}|^2|a_{j}|^2&m=4\\
    \frac{3}{8} \sum_{i=1}^{n}|a_{i}|^{4}+\frac{3}{2} \sum_{1\le i<j \le n}|a_{i}|^2|a_{j}|^2&m\neq 4.
    \end{array}
    \right.
\end{align*}
\end{lemma}
\begin{proof}
  Let $\eta_i=\re (a_{i}\xi_i)=\frac{1}{2}(a_i\xi_i+\ov{a_i\xi_i})$ for $i=1,2,\dots,n$, and so $\eta=\sum_{i=1}^n\eta_i$. The first two moments of $\eta_i$ are
  \begin{align}
    \ex\eta_i&=\frac{1}{2}\left(a_i\ex\xi_i+\ov{a_i}\ex\ov{\xi_i}\right)=0, \label{eq:xi1m} \\
    \ex\eta_i^2& =\ex\left[\frac{(a_i\xi_i+\ov{a_i\xi_i})^2}{4}\right]
   = \ex\left[\frac{a_i^2\xi_i^2}{4}+\frac{a_i\ov{a_i}\xi_i\ov{\xi_i}}{2} +\frac{{\ov{a_i}}^2{\ov{\xi_i}}^2}{4}\right]
   = \frac{a_i^2}{4}\ex \xi_i^2 + \frac{|a_i|^2}{2}+ \frac{{\ov{a_i}}^2}{4}\ex {\ov{\xi_i}}^2
   =\frac{|a_i|^2}{2}, \label{eq:xi2m}
  \end{align}
  where the last equality is due to~\eqref{first-second-moments}.
  Since all the $\eta_i$'s are independent to each other, we have
  \begin{align*}
    \ex\eta&=\ex\left[\sum_{i=1}^n\eta_i\right]=\sum_{i=1}^n\ex\eta_i=0, \\
    \ex\eta^2& =\ex\left(\sum_{i=1}^n\eta_i\right)^2=\ex\left[\sum_{i=1}^n\eta_i^2+2\sum_{1\le i<j\le n}\eta_i\eta_j\right]=\sum_{i=1}^n\ex\eta_i^2 +2\sum_{1\le i<j\le n}\ex\eta_i\ex\eta_j
    =\frac{1}{2}\sum_{i=1}^n|a_i|^2.
  \end{align*}
  Moreover, the fourth moment of $\eta_i$ is
  \begin{align*}
  \ex\eta_i^4&=\ex\left[\frac{1}{16}(a_i\xi_i+\ov{a_i\xi_i})^4\right] \\
     &=\frac{1}{16}\ex \left[a_i^4\xi_i^4 + 4a_i^3\xi_i^3{\ov{a_i}}\ov{\xi_i} + 6a_i^2\xi_i^2{\ov{a_i}}^2{\ov{\xi_i}}^2 + 4a_i\xi_i{\ov{a_i}}^3{\ov{\xi_i}}^3 + {\ov{a_i}}^4{\ov{\xi_i}}^4\right] \\
     &=\frac{1}{16}a_i^4\ex\xi_i^4 + \frac{1}{16}{\ov{a_i}}^4\ex{\ov{\xi_i}}^4 + \frac{1}{4}|a_i|^2 a_i^2\ex\xi_i^2
     +\frac{1}{4}|a_i|^2{\ov{a_i}}^2\ex{\ov{\xi_i}}^2
     +\frac{3}{8}|a_i|^4\\
     &=\left\{
    \begin{array}{ll}
    \frac{1}{16} (a_{i}^4+\ov{a_{i}}^{4})+\frac{3}{8}|a_i|^4&m=4\\
    \frac{3}{8}|a_i|^4&m\neq 4.
    \end{array}
    \right.
  \end{align*}
  Therefore,
  \begin{align*}
  &~~~~\ex\eta^4 \\
  &=\ex\left(\sum_{i=1}^n\eta_i\right)^4\\
  &=\ex\left[\sum_{i=1}^n\eta_i^4+4\sum_{1\le i\neq j\le n}\eta_i\eta_j^3 + 6\sum_{1\le i< j\le n} \eta_i^2\eta_j^2 + 12\sum_{1\le i<j\le n,\,k\neq i, j}\eta_i\eta_j\eta_k^2 + 24\sum_{1\le i< j< k< \ell\le n} \eta_i\eta_j\eta_k\eta_\ell\right] \\
     &=\sum_{i=1}^n\ex\eta_i^4+ 6\sum_{1\le i< j\le n} \ex\eta_i^2 \, \ex\eta_j^2 +\sum_{i=1}^n\ex[\eta_i] \, \ex[p_i(\eta_1,\dots,\eta_{i-1},\eta_{i+1},\dots,\eta_n)] \\
     &=\left\{
    \begin{array}{ll}
    \frac{1}{16} \sum_{i=1}^{n}(a_{i}^4+\ov{a_{i}}^{4})+\frac{3}{8} \sum_{i=1}^{n}|a_{i}|^{4}+\frac{3}{2} \sum_{1\le i<j \le n}|a_{i}|^2|a_{j}|^2&m=4\\
    \frac{3}{8} \sum_{i=1}^{n}|a_{i}|^{4}+\frac{3}{2} \sum_{1\le i<j \le n}|a_{i}|^2|a_{j}|^2&m\neq 4,
    \end{array}
    \right.
  \end{align*}
  where $p_i(\eta_1,\dots,\eta_{i-1},\eta_{i+1},\dots,\eta_n)$ is a cubic polynomial function of $(\eta_1,\dots,\eta_{i-1},\eta_{i+1},\dots,\eta_n)$ for $i=1,2,\dots,n$.
\end{proof}

The following lemma follows straightforwardly from the Berry-Esseen inequality.
\begin{lemma}\label{lem:ineq}
Let $m\ge3$ be an integer or $m=\infty$. If $\bxi=(\xi_1,\xi_2,\dots,\xi_{n})^{\T}\in\C^n$ whose entries are i.i.d.\ uniformly on $\Om_m$, then for any $\ba \in \C^{n}$
$$
\left|\Prob\left\{\frac{\sqrt{2}\,\re(\ba^{\T}\bxi)}{\|\ba\|_2}\le x\right\}-N(x)\right| \le \frac{2\sqrt{2}\,c_0\|\ba\|_\infty}{\|\ba\|_2}.
$$
\end{lemma}
\begin{proof}
Without loss of generality, we assume $a_i\neq 0$ for $i=1,2,\dots,n$ since otherwise we may delete all the zero entries of $\ba$ while keeping values of both sides of the inequality unchanged. Let $\eta_i=\re (a_{i}\xi_i)$ for $i=1,2,\dots,n$. According to~\eqref{eq:xi1m} and~\eqref{eq:xi2m} in the proof of Lemma~\ref{thm:complex}, we have $$\ex \eta_{i}=0 \mbox{ and } \ex\eta_{i}^2=\frac{1}{2}|a_i|^2>0 \mbox{ for } i=1,2,\dots, n.$$
Moreover, for $i=1,2,\dots, n$,
$$\ex|\eta_{i}|^3 = \ex |\re (a_{i}\xi_i)|^3\le |a_i|^3\ex |\xi_i|^3=|a_{i}|^{3}.$$
By applying Lemma~\ref{thm:be}, we get
\begin{align*}
\left|\Prob\left\{\frac{\sqrt{2}\,\re(\ba^{\T}\bxi)}{\|\ba\|_2}\le x\right\}-N(x)\right|
&\le \frac{2\sqrt{2}\,c_0\sum_{i=1}^{n}|a_{i}|^{3}}{\|\ba\|_2^3} \\
&\le \frac{2\sqrt{2}\,c_0\max_{1\le i\le n}\{|a_i|\}\sum_{i=1}^{n}|a_{i}|^2}{\|\ba\|_2^3}  \\
&= \frac{2\sqrt{2}\,c_0\|\ba\|_\infty}{\|\ba\|_2}.
\end{align*}
\end{proof}

We are now ready to present the main probability inequality. The result can be taken as a generalization of~\cite[Lemma 3.2]{KN08} where the case of symmetric Bernoulli random variables is discussed, while ours is the uniform distribution on $\Om_m$ or $\Om_\infty$.
\begin{theorem}\label{thm:key1}
Let $m\ge3$ be an integer or $m=\infty$. If $\bxi=(\xi_1,\xi_2,\dots,\xi_{n})^{\T}\in\C^n$ whose entries are i.i.d.\ uniformly on $\Om_m$, then for any $\ba \in \C^{n}$ and $\delta \in \left(0,\frac{1}{16}\right)$, there exists a constant $c_1(\delta)>0$ such that
$$
\Prob\left\{\re(\ba^{\T}\bxi)\ge \sqrt{\frac{\delta \ln n}{n}}\|\ba\|_1\right\} \ge \frac{c_1(\delta)}{c_2(m)n^{5\delta}},
$$
where $c_2(m):=\min\{k\ge2:k \mbox{ is a divisor of } m\}\le m$, in particular, $c_2(\infty)=2$.
\end{theorem}
\begin{proof}
Denote $\eta=\re(\ba^{\T}\bxi)$. We first prove that there is a constant $n_0(\delta)>0$, depending only on $\delta$, such that the inequality holds when $n\ge n_0(\delta)$ in the following two cases.

In the first case we assume that
\begin{equation}\label{eq:case1}
\|\ba\|_1\le n^{4\delta+\frac{1}{4}}\|\ba\|_2.
\end{equation}
By Lemma~\ref{thm:complex}, it is straightforward to verify that
$$3(\ex\eta^2)^2=\frac{3}{4}\left(\sum_{i=1}^n|a_i|^2\right)^2\ge \max\left\{\frac{1}{16} \sum_{i=1}^{n}(a_{i}^4+\ov{a_{i}}^{4}),0\right\}+\frac{3}{8} \sum_{i=1}^{n}|a_{i}|^{4}+\frac{3}{2} \sum_{1\le i<j \le n}|a_{i}|^2|a_{j}|^2 \ge \ex\eta^4.$$
By the Paley-Zygmund inequality, for any $t \in (0,1)$, we have
\begin{equation}\label{eq:pz}
  \Prob\left\{\eta^2\ge t\ex\eta^2\right\}\ge (1-t)^2\frac{(\ex\eta^2)^2}{\ex \eta^{4}} \ge \frac{(1-t)^2}{3}.
\end{equation}
Given any $t\ge 0$, consider the event $\left\{\eta^2\ge t^2\right\}=\left\{\eta\ge t \right\} \cup \left\{-\eta\ge t \right\}$. If $m$ is even or $m=\infty$, $\Om_m$ is central symmetric and so $\eta=\re(\ba^{\T}\bxi)$ is also symmetric, leading to
\begin{equation}\label{eq:squarebound}
\Prob\left\{\eta^2\ge t^2\right\} = 2\,\Prob\left\{\eta\ge t \right\} \le 3\,\Prob\left\{\eta\ge\frac{t}{2} \right\}.
\end{equation}
If $m$ is odd, it is obvious that both $\omega_m^{\frac{m+1}{2}}\bxi$ and $\omega_m^{\frac{m-1}{2}}\bxi$ have the same distribution to $\bxi$ as $\omega_m^{\frac{m\pm1}{2}}\in\Om_m$, and so both $\re(\ba^{\T}(\omega_m^{\frac{m+1}{2}}\bxi))$ and $\re(\ba^{\T}(\omega_m^{\frac{m-1}{2}}\bxi))$ have the same distribution to $\eta$. By noticing
$$
\re(\ba^{\T}(\omega_m^{\frac{m+1}{2}}\bxi))+\re(\ba^{\T}(\omega_m^{\frac{m-1}{2}}\bxi))
=\re((\omega_m^{\frac{m+1}{2}}+\omega_m^{\frac{m-1}{2}})\ba^{\T}\bxi)
=\re\left(\left(-2\cos\frac{\pi}{m}\right)\ba^{\T}\bxi\right)=-2\eta\cos\frac{\pi}{m}
$$
and $\cos\frac{\pi}{m}\ge\cos\frac{\pi}{3}=\frac{1}{2}$, we have
\begin{align*}
  \Prob\left\{-\eta\ge t \right\}&\le \Prob\left\{-2\eta\cos\frac{\pi}{m}\ge t \right\}\\
  &= \Prob\left\{\re(\ba^{\T}(\omega_m^{\frac{m+1}{2}}\bxi))+\re(\ba^{\T}(\omega_m^{\frac{m-1}{2}}\bxi))\ge t \right\}\\
  &\le \Prob\left\{\re(\ba^{\T}(\omega_m^{\frac{m+1}{2}}\bxi))\ge \frac{t}{2} \right\} +\Prob\left\{\re(\ba^{\T}(\omega_m^{\frac{m-1}{2}}\bxi))\ge \frac{t}{2} \right\} \\
  &=2\,\Prob\left\{\eta\ge \frac{t}{2} \right\}.
\end{align*}
Together with the obvious fact that $\Prob\left\{\eta\ge t \right\}\le \Prob\left\{\eta\ge\frac{t}{2}\right\}$, we arrive at
$$\Prob\left\{\eta^2\ge t^2\right\}= \Prob\left\{\eta\ge t \right\} +\Prob\left\{-\eta\ge t \right\}\le 3 \,\Prob\left\{\eta \ge \frac{t}{2}\right\}.$$
We conclude that~\eqref{eq:squarebound} holds when $m\ge3$ is an integer or $m=\infty$.
%
As $\delta \in (0,\frac{1}{16})$, we can define $n_1(\delta):=\min\left\{n>0: \frac{8\delta \ln n}{n^{\frac{1}{2}-8\delta}}\le\frac{1}{2} \right\}$. Therefore, when $n\ge n_1(\delta)$, it follows from~\eqref{eq:squarebound}  that
\begin{align*}
\Prob\left\{\eta\ge \sqrt{\frac{\delta \ln n}{n}}\|\ba\|_1\right\}&\ge \frac{1}{3}\,\Prob\left\{\eta^2\ge \frac{4\delta \ln n}{n}\|\ba\|_1^2\right\}\\
&\ge\frac{1}{3}\,\Prob\left\{\eta^2\ge \frac{4\delta \ln n}{n^{\frac{1}{2}-8\delta}}\|\ba\|_2^2\right\}\\
&=\frac{1}{3}\,\Prob\left\{\eta^2\ge \frac{8\delta \ln n}{n^{\frac{1}{2}-8\delta}}\ex\eta^2\right\}\\
&\ge\frac{1}{3}\,\Prob\left\{\eta^2\ge \frac{1}{2}\ex\eta^2\right\}\\
&\ge\frac{1}{3}\cdot \frac{1}{3}\cdot \left(1-\frac{1}{2}\right)^2 =\frac{1}{36},
\end{align*}
where the second inequality is due to~\eqref{eq:case1} and the last inequality is due to~\eqref{eq:pz}.

In the second case we assume that
\begin{equation}\label{eq:case2}
\|\ba\|_1 > n^{4\delta+\frac{1}{4}}\|\ba\|_2.
\end{equation}
Let $I=\left\{i\in\{1,2,\dots,n\}: |a_i|\le\frac{2\|\ba\|_2^2}{\|\ba\|_1}\right\}$ and define $\zeta:=\re\left(\sum_{i \in I}a_i\xi_i\right)$. It holds that
$$
\|\ba\|_1=\sum_{i \notin I}\frac{|a_i|^2}{|a_i|}+ \sum_{i \in I}|a_i|\le\frac{\|\ba\|_1}{2\|\ba\|_2^2} \sum_{i \notin I}|a_i|^2+\sqrt{|I|\sum_{i \in I}|a_i|^2}\le \frac{\|\ba\|_1}{2}+\sqrt{n\sum_{i \in I}|a_i|^2},
$$
implying that
\begin{equation}\label{ineq:par}
\frac{\|\ba\|_1}{2\sqrt{n}}\le \sqrt{\sum_{i\in I}|a_i|^2}= \sqrt{2\ex\zeta^2}.
\end{equation}
As $c_2(m)\ge2$ is a divisor of $m$, $\omega_0:=e^{\ii\frac{2\pi}{c_2(m)}}\in\Om_m$ and $\sum_{k=1}^{c_2(m)}\omega_0^k=0$, implying that
$$
\sum_{k=1}^{c_2(m)}\re\left(\omega_0^k\sum_{i \notin I}a_i\xi_i+\sum_{i \in I}a_i\xi_i\right)= c_2(m)\re\left(\sum_{i \in I}a_i\xi_i\right) + \re\left(\left(\sum_{k=1}^{c_2(m)}\omega_0^k\right)\sum_{i \notin I}a_i\xi_i\right)=
c_2(m)\zeta.
$$
For any $t\in \R$, if $\sum_{k=1}^{c_2(m)}\re\left(\omega_0^k\sum_{i \notin I}a_i\xi_i+\sum_{i \in I}a_i\xi_i\right)\ge c_2(m)t$, then there must exist some $k\in\{1,2,\dots,c_2(m)\}$ such that $\re\left(\omega_0^k\sum_{i \notin I}a_i\xi_i+\sum_{i \in I}a_i\xi_i\right)\ge t $. Therefore
\begin{align*}
\Prob\left\{\zeta\ge t\right\}&=\Prob\left\{\sum_{k=1}^{c_2(m)}\re\left(\omega_0^k\sum_{i \notin I}a_i\xi_i+\sum_{i \in I}a_i\xi_i\right)\ge c_2(m)t\right\} \\
&\le\sum_{k=1}^{c_2(m)}\Prob\left\{\re\left(\omega_0^k\sum_{i \notin I}a_i\xi_i+\sum_{i \in I}a_i\xi_i\right)\ge t\right\} \\
&= c_2(m) \Prob\left\{\eta\ge t\right\},
\end{align*}
where the last equality holds because $\re\left(\omega_0^k\sum_{i \notin I}a_i\xi_i+\sum_{i \in I}a_i\xi_i\right)$ has the exact same distribution to $\eta$ as $\omega_0^k\in\Om_m$ for $k=1,2,\dots,c_2(m)$. By letting $t=\sqrt{\frac{\delta \ln n}{n}}\|\ba\|_1$ in the above, we arrive at

\begin{align*}
\Prob\left\{\eta \ge \sqrt{\frac{\delta  \ln n}{n}}\|\ba\|_1\right\}&\ge \frac{1}{c_2(m)}\Prob\left\{\zeta\ge \sqrt{\frac{\delta \ln n}{n}}\|\ba\|_1\right\}\\
&=\frac{1}{c_2(m)}\Prob\left\{\frac{\zeta}{\sqrt{\ex\zeta^2}}\ge \sqrt{\frac{\delta\ln n}{n}}\frac{\|\ba\|_1}{\sqrt{\ex\zeta^2}}\right\}\\
&\ge \frac{1}{c_2(m)}\Prob\left\{\frac{\zeta}{\sqrt{\ex \zeta^2}}\ge \sqrt{8\delta\ln n}\right\}\\
&\ge \frac{1}{c_2(m)}\left(1-N\left(\sqrt{8\delta\ln n}\right)-\frac{2\sqrt{2}c_0\max_{i\in I}|a_i|}{\sqrt{\sum_{i\in I}|a_i|^2}}\right)\\
&\ge \frac{1}{c_2(m)}\left(\int_{\sqrt{8\delta\ln n}}^{\sqrt{8\delta\ln n}+1} \frac{1}{\sqrt{2\pi}}e^{-\frac{x^2}{2}}dx -\frac{8 c_0\sqrt{2n}\|\ba\|_2^2}{\|\ba\|_1^2} \right)\\
&\ge\frac{1}{c_2(m)}\left(\frac{1}{\sqrt{2\pi}}e^{-\frac{(\sqrt{8\delta\ln n}+1)^2}{2}}-\frac{8\sqrt{2} c_0}{n^{8\delta}}\right)\\
&\ge\frac{1}{c_2(m)}\left(\frac{1}{\sqrt{2\pi}n^{5\delta}}-\frac{8\sqrt{2}c_0}{n^{8\delta}}\right) \ge\frac{1}{3c_2(m)n^{5\delta}}
\end{align*}
for $n\ge n_2(\delta):=\min\left\{n>0: \frac{(\sqrt{8\delta\ln n}+1)^2}{2} \le 5\delta\ln n, ~\frac{1}{\sqrt{2\pi}n^{5\delta}}-\frac{8 \sqrt{2}c_0}{n^{8\delta} }\ge \frac{1}{3n^{5\delta}}\right\}$, where the second, third, fourth and fifth inequalities are due to~\eqref{ineq:par}, Lemma~\ref{lem:ineq},~\eqref{ineq:par} and~\eqref{eq:case2}, respectively.

To conclude the proof, it remains to settle the case for $n\le n_0(\delta)=\max\{n_1(\delta),n_2(\delta)\}$. For $i\in\{1,2,\dots,n\}$, consider the set $\Theta_i=\left\{z\in \Om_\infty:|\arg z-\arg a_i|\le\frac{\pi}{3}\right\}$. We have
$$
\Prob\left\{\xi_i\in\Theta_i\right\}\ge\left\{
\begin{array}{ll}
    \frac{\lfloor m/3\rfloor}{m}\ge\frac{1}{5}&m\ge3\\
    \frac{1}{3} & m=\infty,
    \end{array}
\right.
$$
and $\xi_i\in\Theta_i$ implies that $\re(a_i\xi_i)\ge|a_i|\cos\frac{\pi}{3}=\frac{|a_i|}{2}$. Therefore
$\Prob\left\{\re(a_i\xi_i)\ge\frac{|a_i|}{2}\right\}\ge \frac{1}{5}$. By the independence of $\xi_i$'s, we have
$$\Prob\left\{\eta\ge \sqrt{\frac{\delta \ln n}{n}}\|\ba\|_1\right\}\ge \Prob\left\{\eta\ge\frac{\|\ba\|_1}{2}\right\}\ge\prod_{i=1}^n \Prob\left\{\re(a_i\xi_i)\ge\frac{|a_i|}{2}\right\} \ge\frac{1}{5^n}\ge\frac{1}{5^{n_0(\delta)}}.$$

To summarize, for any $\delta \in \left(0,\frac{1}{16}\right)$, there exists $n_0(\delta)>0$, such that
$$
\Prob\left\{\eta\ge \sqrt{\frac{\delta \ln n}{n}}\|\ba\|_1\right\}\ge\left\{
    \begin{array}{ll}
    \min\left\{\frac{1}{3c_2(m)n^{5\delta}},\frac{1}{36}\right\} & n\ge n_0(\delta)\\
    \frac{1}{5^{n_0(\delta)}}& n< n_0(\delta).
    \end{array}
    \right.
$$
Define $c_1(\delta):=\frac{1}{36\cdot 5^{n_0(\delta)}}$ and the lower bound $\frac{c_1(\delta)}{c_2(m)n^{5\delta}}$ holds for all $n$.
\end{proof}
%

Theorem~\ref{thm:key1} provides a lower bound for the random sampling on $\Om_m^n$. For the random sampling on the complex sphere $\BS^n$, we have the following inequality, which is analogous to Theorem~\ref{thm:key1}.
\begin{theorem}\label{thm:key2}
If $\bxi$ is a uniform distribution on $\BS^n$, then for any $\ba \in \C^{n}$ and $\gamma>0$ with $\gamma\ln n<n$, there exists a constant $c_3(\gamma)>0$, such that
$$
\Prob\left\{\re(\ba^{\T}\bxi)\ge \sqrt{\frac{\gamma\ln n}{n}} \| \ba\|_2\right\}\ge \frac{c_{3}(\gamma)}{n^{2\gamma}\sqrt{\ln n}}.
$$
\end{theorem}
The proof is similar to that of~\cite[Lemma 2.5]{HJLZ14}, and is left to interested readers.

\section{Complex multilinear form optimization}\label{sec:mulform}

In this section, we study approximation algorithms for the following complex multilinear form optimization models,
$$
\begin{array}{lll}
(L_m) & \max & \re F(\bx^1,\bx^2,\dots,\bx^d) \\
             & \st & \bx^k \in \Om_m^{n_k} ,\, k=1,2,\dots,d; \\
(L_\infty) & \max  & \re F(\bx^1,\bx^2,\dots,\bx^d) \\
           & \st & \bx^k \in \Om_\infty^{n_k} ,\, k=1,2,\dots,d; \\
(L_S) & \max      & \re F(\bx^1,\bx^2,\dots,\bx^d) \\
      & \st & \bx^k \in \BS^{n_k} ,\, k=1,2,\dots,d,
\end{array}
$$
where $F$ is a complex multilinear form associated with a complex tensor $\F\in\C^{n_1\times n_2\times \dots \times n_d}$. Without loss of generality, we assume that $n_1\le n_2\le\dots\le n_d$ in this section.

All these models were studied by Jiang et al.~\cite{JLZ14}. However, the algorithms in this section improve that in~\cite{JLZ14} in terms of approximation ratios. The main ingredients of our algorithms are recursions on the degree of the multilinear form and random sampling on the constraint sets.

\subsection{Multilinear form in the $m$-th roots of unity or the complex unit circle}\label{sec:mform1}

In this subsection, we discuss the discrete optimization model $(L_m)$ and the continuous one $(L_\infty)$ together as the main ideas are similar. When $d=2$, both $(L_m)$ and $(L_\infty)$ are already NP-hard. Huang and Zhang~\cite{HZ10} studied these two models for $d=2$ via semidefinite program relaxation and proposed polynomial-time randomized algorithms with constant worst-case approximation ratios $c_4(m):=0.7118\cos^2\frac{\pi}{m}$ for $(L_m)$ and $0.7118$ for $(L_\infty)$ which coincides that of $(L_m)$ when $m\rightarrow\infty$.
For $d\ge3$, Jiang et al.~\cite{JLZ14} applied some decomposition routines and proposed randomized algorithms with approximation ratios $\left(\frac{m^2}{2\pi}\sin^2\frac{\pi}{m}-1\right) \left(\frac{m^2}{4\pi}\sin^2\frac{\pi}{m}\right)^{d-2} \left(\prod_{k=1}^{d-2}n_k\right)^{-\frac{1}{2}}$ for $(L_m)$ and $c_4(\infty)\left(\frac{\pi}{4}\right)^{d-2}\left(\prod_{k=1}^{d-2}n_k\right)^{-\frac{1}{2}}$ for $(L_\infty)$.
With the help of Theorem~\ref{thm:key1}, we manage to provide an improved randomized approximation algorithm, which can be applied to both $(L_m)$ and $(L_\infty)$.

\begin{algorithm}\label{alg:1}
A polynomial-time randomized algorithm of $(L_m)$ when $m\ge3$ is an integer or $m=\infty$:
\begin{enumerate}
\item Randomly generate $\bxi^k\in\Om_m^{n_k}$ for $k=1,2,\dots,d-2$, where all $\xi^k_i$'s are i.i.d.\ uniformly on $\Om_m$;
\item Apply the approximation algorithm in~\cite{HZ10} to solve the bilinear form optimization problem
$$
\begin{array}{ll}
\max & \re F(\bxi^1,\bxi^2,\dots,\bxi^{d-2},\bx^{d-1},\bx^d ) \\
\st &  \bx^{d-1}\in \Om_m^{n_{d-1}}, \, \bx^d \in \Om_m^{n_d},
\end{array}
$$
and get its approximate solution $(\bxi^{d-1},\bxi^d)$;
\item Compute an objective value $\re F(\bxi^1,\bxi^2,\dots,\bxi^d )$;
\item Repeat the above procedures independently $\ln\frac{1}{\epsilon} \left(\frac{c_2(m)}{c_1(\delta)}\right)^{d-2} \prod_{k=1}^{d-2} n_k^{5\delta}$ times for any given $\epsilon>0$ and $\delta \in (0,\frac{1}{16})$, and choose a solution with the largest objective value.
\end{enumerate}
\end{algorithm}

\begin{theorem}\label{thm:lm}
If $m\ge3$ is an integer or $m=\infty$, Algorithm~\ref{alg:1} solves $(L_{m})$ with an approximation ratio $c_4(m)\delta^{\frac{d-2}{2}}\left(\prod_{k=1}^{d-2}\frac{\ln n_k}{n_k}\right)^{\frac{1}{2}}$, i.e., for any given $\epsilon>0$ and $\delta \in (0,\frac{1}{16})$, a feasible solution $(\by^1,\by^2,\dots,\by^d)$ can be generated in polynomial time with probability at least $1-\epsilon$, such that $$\re F(\by^1,\by^2,\dots, \by^d )\ge c_4(m)\delta^{\frac{d-2}{2}}\left(\prod_{k=1}^{d-2}\frac{\ln n_k}{n_k}\right)^{\frac{1}{2}} v_{\max}(L_m).$$
\end{theorem}
\begin{proof}
Suppose $(\bxi^1,\bxi^2,\dots,\bxi^d)$ is an approximate solution generated by the first three steps of Algorithm~\ref{alg:1}, i.e., without repeated sampling and choosing the best one. For any $t~(2\le t\le d)$, we treat $(\bxi^1,\bxi^2,\dots,\bxi^{d-t})$ as given parameters and define the following problem
\[
\begin{array}{lll}
  (F_t)   & \max   & \re F(\bxi^1,\bxi^2,\dots,\bxi^{d-t},\bx^{d-t+1},\bx^{d-t+2}\dots,\bx^d)\\
       & \st & \bx^k\in\Om_m^{n_k},\,k=d-t+1,d-t+2,\dots,d.
\end{array}
\]
By applying the first three steps of Algorithm~\ref{alg:1} to $(F_t)$, we get a randomly generated feasible solution $(\bxi^{d-t+1},\bxi^{d-t},\dots,\bxi^d)$ of $(F_t)$ to update the previous $\bxi^k$'s for $d-t+1\le k\le d$. In the remaining, we prove by induction on $t$ that for each $t=2,3,\dots,d$,
\begin{align}
&~~~~\mathop{\Prob}_{(\bxi^{d-t+1},\bxi^{d-t+2},\dots,\bxi^d )}\left\{\re F(\bxi^1,\bxi^2,\dots,\bxi^d)\ge
c_4(m)\delta^{\frac{t-2}{2}} \left(\prod_{k=d-t+1}^{d-2}\frac{\ln n_k}{n_k}\right)^{\frac{1}{2}} v_{\max}(F_t)\right\} \nonumber \\
&\ge \frac{(c_1(\delta))^{t-2}}{(c_2(m))^{t-2}\prod_{k=d-t+1}^{d-2} n_k^{5\delta}}. \label{eq:lmproof}
\end{align}
In other words, $(\bxi^{d-t+1},\bxi^{d-t+2},\dots,\bxi^d)$ is a $c_4(m)\delta^{\frac{t-2}{2}} \left(\prod_{k=d-t+1}^{d-2}\frac{\ln n_k}{n_k}\right)^{\frac{1}{2}}$-approximate solution of $(F_t)$ with a nontrivial probability.

For the base case $t=2$, the algorithm by Huang and Zhang~\cite{HZ10} (the second step of Algorithm~\ref{alg:1}) guarantees a constant ratio $c_4(m)$, i.e., $\re F(\bxi^1,\bxi^2,\dots,\bxi^d) \ge c_4(m)\,v_{\max}(F_2)$, implying~\eqref{eq:lmproof}. Suppose now~\eqref{eq:lmproof} holds for $t-1$. To prove that~\eqref{eq:lmproof} holds for $t$, we notice that $(\bxi^1,\bxi^2,\dots,\bxi^{d-t})$ are given fixed parameters. Denote $(\bz^{d-t+1},\bz^{d-t+2},\dots,\bz^d)$ to be an optimal solution of $(F_t)$, and define the following two events
\begin{align*}
\Phi_1&=\left\{\bxi^{d-t+1} \in \Om_m^{n_{d-t+1}} : \re F(\bxi^1,\dots,\bxi^{d-t},\bxi^{d-t+1},\bz^{d-t+2},\dots,\bz^d)\ge  \sqrt{\frac{\delta\ln n_{d-t+1}}{n_{d-t+1}}} v_{\max}(F_t) \right\}, \\
\Phi_2&=\Bigg\{\bxi^{d-t+1}\in \Phi_1,\bxi^{d-t+2}\in\Om_m^{n_{d-t+2}},\dots,\bxi^d \in\Om_m^{n_d} :\\
& ~~~~\re F(\bxi^1,\bxi^2,\dots,\bxi^d)\ge c_4(m)\delta^{\frac{t-3}{2}}\left(\prod_{k=d-t+2}^{d-2}\frac{\ln n_k}{n_k}\right)^{\frac{1}{2}}\re F(\bxi^1,\dots,\bxi^{d-t},\bxi^{d-t+1},\bz^{d-t+2},\dots,\bz^d)
\Bigg\}.
\end{align*}
Clearly we have
\begin{align}
&~~~~\mathop{\Prob}_{(\bxi^{d-t+1},\bxi^{d-t+2},\dots,\bxi^d)}\left\{\re F(\bxi^1,\bxi^2,\dots,\bxi^d)\ge
c_4(m) \delta^{\frac{t-2}{2}} \left(\prod_{k=d-t+1}^{d-2}\frac{\ln n_k}{n_k}\right)^{\frac{1}{2}} v_{\max}(F_t)\right\}\nonumber\\
&\ge\mathop{\Prob}_{(\bxi^{d-t+1},\bxi^{d-t+2},\dots,\bxi^d)}\left\{(\bxi^{d-t+1},\bxi^{d-t+2},\dots,\bxi^d)\in \Phi_2\Big{|}\,\bxi^{d-t+1} \in \Phi_1 \right\}\cdot \mathop{\Prob}_{\bxi^{d-t+1}}\left\{ \bxi^{d-t+1} \in \Phi_1 \right\}.\label{eq:star}
\end{align}
To lower bound~\eqref{eq:star}, first we notice that $(\bz^{d-t+2},\bz^{d-t+3},\dots,\bz^d)$ is a feasible solution of $(F_{t-1})$, and so $\re F(\bxi^1,\dots,\bxi^{d-t},\bxi^{d-t+1},\bz^{d-t+2},\dots,\bz^d)\le v_{\max}(F_{t-1})$, which leads to
\begin{align*}
&~~~~\mathop{\Prob}_{(\bxi^{d-t+1},\bxi^{d-t+2},\dots,\bxi^d)}\left\{(\bxi^{d-t+1},\bxi^{d-t+2},\dots,\bxi^d)\in \Phi_2 \, \Big{|}\,\bxi^{d-t+1} \in \Phi_1 \right\}\\
&\ge\mathop{\Prob}_{(\bxi^{d-t+1},\bxi^{d-t+2},\dots,\bxi^d)}\left\{\re F(\bxi^1,\bxi^2,\dots,\bxi^d)\ge  c_4(m)\delta^{\frac{t-3}{2}}\left(\prod_{k=d-t+2}^{d-2}\frac{\ln n_k}{n_k}\right)^{\frac{1}{2}} v_{\max}(F_{t-1})\, \Bigg{|}\,\bxi^{d-t+1 } \in \Phi_1\right\}\\
&\ge \frac{(c_1(\delta))^{t-3}}{(c_2(m))^{t-3}\prod_{k=d-t+2}^{d-2} n_k^{5\delta}},
\end{align*}
where the last inequality is due to the induction assumption on $t-1$. Second, we have
\begin{align*}
&~~~~\mathop{\Prob}_{\bxi^{d-t+1}}\left\{ \bxi^{d-t+1} \in \Phi_1 \right\}\\
&=\mathop{\Prob}_{\bxi^{d-t+1}}\left\{\re F(\bxi^1,\dots,\bxi^{d-t+1},\bz^{d-t+2},\dots,\bz^d)\ge \sqrt{\frac{\delta \ln n_{d-t+1}}{n_{d-t+1}}} \re F(\bxi^1,\dots,\bxi^{d-t},\bz^{d-t+1},\dots,\bz^d)\right\}\\
&\ge \mathop{\Prob}_{\bxi^{d-t+1}}\left\{\re F(\bxi^1,\dots,\bxi^{d-t+1},\bz^{d-t+2},\dots,\bz^d)\ge  \sqrt{\frac{\delta \ln n_{d-t+1}}{n_{d-t+1}}} \left\|F(\bxi^1,\dots,\bxi^{d-t},\bullet,\bz^{d-t+2},\dots,\bz^d)\right\|_1\right\}\\
&\ge \frac{c_1(\delta)}{c_2(m)n_{d-t+1}^{5\delta}},
\end{align*}
where first inequality is due to the fact that
$$
\re F(\bxi^1,\dots,\bxi^{d-t},\bz^{d-t+1},\bz^{d-t+2},\dots,\bz^d)
\le \left\|F(\bxi^1,\dots,\bxi^{d-t},\bullet,\bz^{d-t+2},\dots,\bz^d)\right\|_1
$$
and the last inequality is due to Theorem~\ref{thm:key1}. With the above two bounds, we can lower bound the right hand side of~\eqref{eq:star}, and conclude
\begin{align*}
&~~~~\mathop{\Prob}_{(\bxi^{d-t+1},\bxi^{d-t+2},\dots,\bxi^d)}\left\{\re F(\bxi^1,\bxi^2,\dots,\bxi^d)\ge
c_4(m)\delta^{\frac{t-2}{2}}\left( \prod_{k=d-t+1}^{d-2}\frac{\ln n_k}{n_k}\right)^{\frac{1}{2}} v_{\max}(F_t)\right\} \\
&\ge \frac{(c_1(\delta))^{t-3}}{(c_2(m))^{t-3}\prod_{k=d-t+2}^{d-2} n_k^{5\delta}} \cdot \frac{c_1(\delta)}{c_2(m)n_{d-t+1}^{5\delta}} =  \frac{(c_1(\delta))^{t-2}}{(c_2(m))^{t-2}\prod_{k=d-t+1}^{d-2} n_k^{5\delta}}.
\end{align*}

Since $(F_d)$ is exactly $(L_m)$, the first three steps of Algorithm~\ref{alg:1} can generate an approximate solution of $(L_m)$ with approximation ratio $c_4(m)\delta^{\frac{d-2}{2}}\left(\prod_{k=1}^{d-2}\frac{\ln n_k}{n_k}\right)^{\frac{1}{2}}$ and with probability at least $\frac{(c_1(\delta))^{d-2}}{(c_2(m))^{d-2}\prod_{k=1}^{d-2} n_k^{5\delta}}:=\theta$. By applying the last step of Algorithm~\ref{alg:1}, if we independently draw
$$
\ln\frac{1}{\epsilon}\left(\frac{c_2(m)}{c_1(\delta)}\right)^{d-2} \prod_{k=1}^{d-2} n_k^{5\delta}
=\frac{\ln\frac{1}{\epsilon}}{\theta}
$$
trials and choose a solution with the largest objective value, then the probability of success is at least $1-(1-\theta)^{\frac{\ln\frac{1}{\epsilon}}{\theta}}\ge1-\epsilon$.
\end{proof}

\subsection{Multilinear form with the spherical constraint}\label{sec:mform3}

Let us now consider the model $(L_S)$. This problem is also known to be the largest singular value of a high order complex tensor~\cite{L05}. When the order of a tensor, $d=2$, $(L_S)$ is to compute the largest singular value of a complex matrix, which can be done in polynomial-time via the singular value decomposition. For general order $d$, Jiang et al.~\cite{JLZ14} introduced a deterministic polynomial-time algorithm with approximation ratio $\left(\prod_{k=1}^{d-2}n_k\right)^{-\frac{1}{2}}$ via tensor relaxation, a complex extension of the method proposed in~\cite{HLZ10}. With the help of Theorem~\ref{thm:key2}, we now propose a random sampling based polynomial-time algorithm with improved approximation ratio, comparable to Algorithm~\ref{alg:1}.

\begin{algorithm}\label{alg:3}
A polynomial-time randomized algorithm of $(L_S)$:
\begin{enumerate}
\item Randomly and independently generate $\bxi^k$ uniformly on $\BS^{n_k}$ for $k=1,2,\dots,d-2$;
\item Find the left singular vector $\bxi^{d-1}\in\BS^{n_{d-1}}$ and the right singular vector $\bxi^d\in\BS^{n_d}$ corresponding to the largest singular value of the complex matrix $F(\bxi^1,\bxi^2,\dots,\bxi^{d-2},\bullet,\bullet)$, i.e., obtain an optimal solution $(\bxi^{d-1},\bxi^d)$ of the bilinear form optimization problem
    $$
    \begin{array}{ll}
    \max & \re F(\bxi^1,\bxi^2,\dots,\bxi^{d-2},\bx^{d-1},\bx^d ) \\
    \st &  \bx^{d-1}\in \BS^{n_{d-1}}, \, \bx^d \in \BS^{n_d};
    \end{array}
    $$
\item Compute an objective value $\re F(\bxi^1,\bxi^2,\dots,\bxi^d )$;
\item Repeat the above procedures independently $\frac{\ln \frac{1}{\epsilon}}{(c_3(\gamma))^{d-2}} \prod_{k=1}^{d-2} n_k^{2\gamma}\sqrt{\ln n_k}$ times for any given $\epsilon>0$ and $\gamma \in (0,\frac{n_1}{\ln n_1})$, and choose a solution with the largest objective value.
\end{enumerate}
\end{algorithm}

We have the following approximation result for the problem $(L_S)$, which improves the approximation ratio studied in~\cite{JLZ14}. Its proof is similar to that of Theorem~\ref{thm:lm} by applying a recursive procedure. The main difference to that of Theorem~\ref{thm:lm} is that the probability bound in Theorem~\ref{thm:key2} replaces the one in Theorem~\ref{thm:key1}. We left the exercise to interested readers.

\begin{theorem}\label{theo:mulsph}
Algorithm~\ref{alg:3} solves $(L_{S})$ with an approximation ratio $\gamma^{\frac{d-2}{2}} \left(\prod_{k=1}^{d-2}\frac{\ln n_k}{n_k}\right)^{\frac{1}{2}}$, i.e., for any given $\epsilon>0$ and $\gamma \in (0,\frac{n_1}{\ln n_1})$, a feasible solution $(\by^1,\by^2,\dots,\by^d)$ can be generated in polynomial time with probability at least $1-\epsilon$, such that $$\re F(\by^1,\by^2,\dots, \by^d )\ge \gamma^{\frac{d-2}{2}} \left(\prod_{k=1}^{d-2}\frac{\ln n_k}{n_k}\right)^{\frac{1}{2}} v_{\max}(L_S).$$
\end{theorem}

We remark that the approximation ratio in Theorem~\ref{theo:mulsph} is the same as that of~\cite[Theorem 4.3]{HJLZ14} for the real case of $(L_S)$. For a general complex model $(L_S)$, it is an obvious but tedious way to rewrite $(L_S)$ as a real model by doubling its decision variables, and directly apply the result of~\cite[Theorem 4.3]{HJLZ14} to get an approximation ratio $\gamma^{\frac{d-2}{2}} \left(\prod_{k=1}^{d-2}\frac{\ln 2n_k}{2n_k}\right)^{\frac{1}{2}}$, which is obviously worse than that of Theorem~\ref{theo:mulsph} where a complex random sampling approach is applied directly.

\section{General conjugate form optimization}\label{sec:gform}

With all the preparations ready, we are now able to study general conjugate form optimization models,
$$
\begin{array}{lll}
(G_m) & \max & g(\bx) \\
             & \st & \bx \in \Om_m^n; \\
(G_\infty) & \max  & g(\bx) \\
           & \st & \bx \in \Om_\infty^n; \\
(G_S) & \max      & g(\bx) \\
      & \st & \bx \in \BS^n.
\end{array}
$$
In the above models,
\begin{equation}\label{eq:gmultilinear}
  g(\bx)=G\bigg(\underbrace{\binom{\ov\bx}{\bx},\binom{\ov\bx}{\bx},\dots,\binom{\ov\bx}{\bx}}_d\bigg)
\end{equation}
is a real-valued general conjugate form of $\bx\in\C^n$ associated with a conjugate super-symmetric tensor $\G\in\C^{(2n)^d}$.

These complex optimization problems were studied by Jiang et al.~\cite{JLZ14} for a special class of $g(\bx)$ where the number of conjugate variables is equal to the number of usual variables in every monomial, i.e., symmetric conjugate forms. When applying the approximation algorithms in this section to this special class of $g(\bx)$, the obtained approximation ratios actually improve that of Jiang et al.~\cite{JLZ14}.

\subsection{Conjugate form in the $m$-th roots of unity or the complex unit circle}\label{sec:gform1}

Due to similarity, we discuss approximation algorithms of $(G_m)$ and $(G_\infty)$ together. First, by noticing~\eqref{eq:gmultilinear} and applying the tensor relaxation method, $(G_m)$ and $(G_\infty)$ can be relaxed to
$$
\begin{array}{lll}
(LG_m) & \max & \re G(\bx^1,\bx^2,\dots,\bx^d) \\
             & \st & \bx^k \in \Om_m^{2n} ,\, k=1,2,\dots,d, \\
(LG_\infty) & \max  & \re G(\bx^1,\bx^2,\dots,\bx^d) \\
           & \st & \bx^k \in \Om_\infty^{2n} ,\, k=1,2,\dots,d,
\end{array}
$$
respectively, which are special cases of $(L_m)$ and $(L_\infty)$ studied in Section~\ref{sec:mulform}, respectivly. Let $m\ge3$ be an integer or $m=\infty$. According to Theorem~\ref{thm:lm}, for any given $\delta \in (0,\frac{1}{16})$, $\bz^1,\bz^2,\dots,\bz^d\in \Om_m^{2n}$ can be generated in polynomial time, such that
$$\re G(\bz^1,\bz^2,\dots, \bz^d)
\ge c_4(m)\left(\frac{\delta\ln (2n)}{2n}\right)^{\frac{d-2}{2}}v_{\max}(LG_m)
\ge c_4(m)\left(\frac{\delta\ln (2n)}{2n}\right)^{\frac{d-2}{2}}v_{\max}(G_m),$$
where the last inequality holds because $(LG_m)$ is a relaxation of $(G_m)$. Let $\bz^k=\binom{\bx^k}{\by^k}$ for $k=1,2,\dots,d$, by the polarization identity in Theorem~\ref{thm:link},
$$
\ex\left[\left(\prod_{i=1}^d \ov{\xi_i}\right)g \left( \sum_{k=1}^d\left(\ov{\xi_k\bx^k}+\xi_k\by^k\right) \right)\right]
=d!\, G\left(\binom{\bx^1}{\by^1},\binom{\bx^2}{\by^2},\dots,\binom{\bx^d}{\by^d}\right)=d!\,G(\bz^1,\bz^2,\dots,\bz^d),
$$
where $\xi_2,\xi_2,\dots,\xi_d$ are i.i.d.\ uniformly on $\Om_m$. By dividing $(2d)^d$ and taking the real part,
\begin{equation} \label{eq:linking2}
\ex\left[\re\left(\prod_{i=1}^d\ov{\xi_i}\right)g \left(\frac{1}{2d} \sum_{k=1}^d\left(\ov{\xi_k\bx^k}+\xi_k\by^k\right) \right)\right]
=\frac{d!}{(2d)^d}\re G(\bz^1,\bz^2,\dots,\bz^d).
\end{equation}
Let us define
\begin{equation} \label{eq:soldef}
  \bu_\xi:=\frac{1}{2d} \sum_{k=1}^d\left(\ov{\xi_k\bx^k}+\xi_k\by^k\right) .
\end{equation}
Therefore,~\eqref{eq:linking2} leads to
\begin{equation} \label{eq:linking3}
\ex\left[\re\left(\prod_{i=1}^d\ov{\xi_i}\right)g (\bu_\xi) \right]
=\frac{d!}{(2d)^d}\re G(\bz^1,\bz^2,\dots,\bz^d) \ge\frac{c_4(m)d!}{(2d)^d}\left(\frac{\delta\ln (2n)}{2n}\right)^{\frac{d-2}{2}}v_{\max}(G_m).
\end{equation}
Observing that $\xi_k$'s and the components of $\bx^k$'s and $\by^k$'s belong to $\Om_m$, every component of $\bu_\xi$ is a convex combination of elements in $\Om_m$ implying that the components of $\bu_\xi$ belong to $\conv(\Om_m)$.

Our next step is to construct a randomized approximate solution of $(L_m)$ from $\bu_\xi$'s. Before randomization and showing its solution quality, we first present some properties of real-valued general conjugate forms.
\begin{proposition}\label{thm:convfeas}
Let $m\ge3$ be an integer or $m=\infty$. Suppose $g(\bx)$ is a real-valued general conjugate form and $\bx \in \C^n$ with $x_i\in\conv(\Om_m)$ for $i=1,2,\dots, n$. \\
(i) If $g(\bx)$ is square-free, i.e., the sum of the powers of $x_i$ and $\ov{x_i}$ is less than two for every $1\le i\le n$ in every monomial, then $\by,\bz\in\Om_m^{n}$ can be found in polynomial time, such that $g(\by)\le g(\bx)\le g(\bz)$.\\
(ii) If $g(\bx)$ is convex, then $\bz\in\Om_m^{n}$ can be found in polynomial time, such that $g(\bx)\le g(\bz)$.
\end{proposition}
\begin{proof}
If $g(\bx)$ is square-free, by fixing $x_2, x_3,\dots,x_n$ as constants and taking $x_1$ as the only variable, we may write
$$g(\bx)= x_1p_1(x_2,x_3,\dots,x_n) + \ov{x_1}p_2(x_2,x_3,\dots,x_n) + p_3(x_2,x_3,\dots,x_n):=p(x_1).$$
As $p(x_1)=g(\bx)$ is real-valued, $p_3(x_2,x_3,\dots,x_n)\in\R$ and $p_2(x_2,x_3,\dots,x_n)=\ov{p_1(x_2,x_3,\dots,x_n)}$, and we have
\begin{align*}
  p(x_1) &= x_1p_1(x_2,x_3,\dots,x_n) + \ov{x_1p_1(x_2,x_3,\dots,x_n)} + p_3(x_2,x_3,\dots,x_n)\\
   &= 2\,\re(x_1p_1(x_2,x_3,\dots,x_n)) + p_3(x_2,x_3,\dots,x_n).
\end{align*}
Therefore, $p(x_1)$ is a linear function of $x_1$, whose optimal value over $\conv(\Om_m)$ is attained at one of its vertices, i.e., $z_1\in\Om_m$ can be found easily such that $p(z_1)\ge p(x_1)$. Now, repeat the same procedures for $x_2,x_3,\dots, x_n$, and let them be replaced by $z_2, z_3,\dots, z_n$, respectively. Then $\bz\in\Om_m^n$ satisfies $g(\bz) \ge g(\bx)$. Using the same argument, we may find $\by\in\Om_m^n$ in polynomial time, such that $g(\by) \le  g(\bx)$. The case that $g(\bx)$ is convex can be proven similarly.
\end{proof}

\begin{proposition}\label{thm:convexlemma}
  If a real-valued general conjugate form is convex, then it is nonnegative.
\end{proposition}
\begin{proof}
Let the real-valued general conjugate form be $g(\bx)$ associated with a conjugate super-symmetric tensor $\G\in\C^{(2n)^d}$. Define $p_{\bx,\by}: \R\to\R$ where $p_{\bx,\by}(t)= g(\bx+t\by)$. Since $g(\bx)$ is convex, it is well know in convex analysis that $p_{\bx,\by}(t)$ is a convex function of $t\in\R$ for all $\bx,\by\in\C^n$. From the tensor representation~\eqref{eq:gmultilinear},
$$
p_{\bx,\by}(t)=g(\bx+t\by)= G\bigg( \underbrace{ \binom{\ov{\bx+t\by}}{\bx+t\by}, \binom{\ov{\bx+t\by}}{\bx+t\by}, \dots, \binom{\ov{\bx+t\by}}{\bx+t\by}}_d \bigg).
$$
Since $\G$ is symmetric, direct computation shows that
$$
p^{\prime}_{\bx,\by}(t)=d\,G\bigg(\binom{\ov\by}{\by}, \underbrace{\binom{\ov{\bx+t\by}}{\bx+t\by},\dots,\binom{\ov{\bx+t\by}}{\bx+t\by}}_{d-1} \bigg),
$$
and furthermore
$$
p^{\prime\prime}_{\bx,\by}(t)=d(d-1) G\bigg(\binom{\ov\by}{\by},\binom{\ov\by}{\by}, \underbrace{\binom{\ov{\bx+t\by}}{\bx+t\by},\dots,\binom{\ov{\bx+t\by}}{\bx+t\by}}_{d-2} \bigg)\ge0
$$
for all $t\in\R$ and $\bx,\by\in\C^n$. In particular, by letting $t=0$ and $\by=\bx$ we get
$$
p^{\prime\prime}_{\bx,\bx}(0)=d(d-1) G\bigg( \underbrace{ \binom{\ov\bx}{\bx}, \binom{\ov\bx}{\bx}, \dots,\binom{\ov\bx}{\bx}}_d \bigg)=d(d-1)\, g(\bx)\ge0
$$
for all $\bx\in\C^n$, proving the nonnegativity of $g(\bx)$.
\end{proof}

We are now able to present our main results in this subsection. The approximation bounds of $(L_m)$ cannot be guaranteed in general without additional conditions of the real-valued general conjugate form $g(\bx)$. Our results below are presented when $g(\bx)$ is either convex or square-free.
\begin{theorem}\label{thm:convex}
Let $m\ge3$ be an integer or $m=\infty$. If $g(\bx)$ is convex, then $(G_m)$ admits a polynomial-time randomized algorithm with approximation ratio $\frac{c_4(m)d!}{(2d)^d} \left(\frac{\delta\ln (2n)}{2n}\right)^{\frac{d-2}{2}}$. 
\end{theorem}
\begin{proof}
According to~\eqref{eq:linking3}, by randomization, we are able to find $\eta_1,\eta_2,\dots,\eta_d\in\Om_m$ in polynomial time, such that
\begin{equation}\label{eq:linking5}
\re\left(\prod_{i=1}^d\ov{\eta_{i}}\right)g (\bu_\eta)
\ge\frac{c_4(m)d!}{(2d)^d}\left(\frac{\delta\ln (2n)}{2n}\right)^{\frac{d-2}{2}}v_{\max}(G_m).
\end{equation}
Since the components of $\bu_\eta$ belong to $\conv(\Om_m)$ and $g(\bx)$ is convex, by Proposition~\ref{thm:convfeas}, $\bz \in \Om_m^{n}$ can be found in polynomial time, such that $g(\bz)\ge g(\bu_\eta)$. Finally, by Proposition~\ref{thm:convexlemma}, $g(\bu_\eta)\ge0$ since $g(\bx)$ is convex, and we get
$$
g(\bz)\ge g(\bu_\eta) = \left|\prod_{i=1}^d\ov{\eta_{i}}\right|g(\bu_\eta)
\ge  \re\left(\prod_{i=1}^d\ov{\eta_{i}}\right)g (\bu_\eta)
\ge\frac{c_4(m)d!}{(2d)^d}\left(\frac{\delta\ln (2n)}{2n}\right)^{\frac{d-2}{2}}v_{\max}(G_m).
$$
\end{proof}

\begin{theorem}\label{thm:squarefree}
Suppose $g(\bx)$ is square-free. \\
(i) If $m\ge3$ is an integer or $m=\infty$, then $(G_m)$ admits a polynomial-time randomized algorithm with relative approximation ratio $\frac{c_4(m)d!}{(2d)^d} \left(\frac{\delta\ln (2n)}{2n}\right)^{\frac{d-2}{2}}$, i.e., for any given $\delta \in (0,\frac{1}{16})$, a feasible solution $\bz\in\Om_m^n$ can be generated in polynomial time, such that
$$g(\bz)-v_{\min}(G_m)\ge\frac{c_4(m)d!}{(2d)^d} \left(\frac{\delta\ln (2n)}{2n}\right)^{\frac{d-2}{2}} \left( v_{\max}(G_m) - v_{\min}(G_m) \right).$$
(ii) If $d$ is odd and $m\ge4$ is an even integer or $m=\infty$, then $(G_m)$ admits a polynomial-time randomized algorithm with approximation ratio $\frac{c_4(m)d!}{(2d)^d} \left(\frac{\delta\ln (2n)}{2n}\right)^{\frac{d-2}{2}}$.
\end{theorem}
\begin{proof}
Let us first prove the second part which is similar to that of Theorem~\ref{thm:convex}. According to~\eqref{eq:linking5}, we can generate $\bu_\alpha$ whose components belong to $\conv(\Om_m)$, such that
$$\re\left(\prod_{i=1}^d\ov{\alpha_{i}}\right)g (\bu_\alpha)  \ge\frac{c_4(m)d!}{(2d)^d}\left(\frac{\delta\ln (2n)}{2n}\right)^{\frac{d-2}{2}}v_{\max}(G_m).$$
Since $m\ge4$ is an even integer or $m=\infty$, $\conv(\Om_m)$ is central-symmetric. The components of $-\bu_\alpha$ belong to $\conv(\Om_m)$, and so as to $\bu_\eta:=\arg\max\{g(\bu_\alpha),g(-\bu_\alpha)\}$. As $d$ is odd, $g(-\bu_\alpha)=-g(\bu_\alpha)$ and so
$$
g(\bu_\eta)=|g(\bu_\alpha)|\ge\re\left(\prod_{i=1}^d\ov{\alpha_{i}}\right)g (\bu_\alpha)  \ge\frac{c_4(m)d!}{(2d)^d}\left(\frac{\delta\ln (2n)}{2n}\right)^{\frac{d-2}{2}}v_{\max}(G_m).
$$
Finally, as $g(\bx)$ is square-free and the components of $\bu_\eta$ belong to $\conv(\Om_m)$, by Proposition~\ref{thm:convfeas}, we can find $\bz\in\Om_m^n$ such that $g(\bz)\ge g(\bu_\eta)$, proving the approximation guarantee for the second part.

Let us now prove the first part in two cases. In the first case we assume that
\begin{equation} \label{eq:gcase1}
  v_{\max}(G_{m})\ge \frac{2}{3}(v_{\max}(G_{m})-v_{\min}(G_{m})).
\end{equation}
Let $\bu_\xi$ be defined in~\eqref{eq:soldef}. Since the components of any $\bu_\xi$ belong to $\conv(\Om_m)$ and $g(\bx)$ is square-free, by Proposition~\ref{thm:convfeas}, we can find $\by_\xi\in\Om_m^n$ such that $g(\bu_\xi)\ge g(\by_\xi)\ge v_{\min}(G_m)$. Moreover, as $\xi_i$'s are i.i.d.\ uniformly on $\Om_m$, it is easy to see that $\prod_{i=1}^d\ov{\xi_i}$ is also a uniform distribution on $\Om_m$, implying that $\ex\left[\prod_{i=1}^d\ov{\xi_i}\right]=0$ and
\begin{equation} \label{eq:xipositive}
  \Prob\left\{\re\left(\prod_{i=1}^d\ov{\xi_i}\right)>0\right\}\le \frac{\frac{m+1}{2}}{m}\le\frac{2}{3}.
\end{equation}
Therefore, by noticing $g (\bu_\xi)-v_{\min}(G_m)\ge0$, we have
\begin{align*}
&~~~~\ex\left[\re\left(\prod_{i=1}^d\ov{\xi_i}\right) g (\bu_\xi) \right] \\
& =\ex\left[\re\left(\prod_{i=1}^d\ov{\xi_i}\right) (g (\bu_\xi)-v_{\min}(G_m)) \right]\\
& = \ex\left[\re\left(\prod_{i=1}^d\ov{\xi_i}\right) (g (\bu_\xi)-v_{\min}(G_m)) \,\Bigg{|}\, \re\left(\prod_{i=1}^d\ov{\xi_i}\right)>0 \right] \Prob\left\{\re\left(\prod_{i=1}^d\ov{\xi_i}\right)>0\right\}\ \\
 & ~~~~ + \ex\left[\re\left(\prod_{i=1}^d\ov{\xi_i}\right) (g (\bu_\xi)-v_{\min}(G_m)) \,\Bigg{|}\, \re\left(\prod_{i=1}^d\ov{\xi_i}\right)\le 0 \right] \Prob\left\{\re\left(\prod_{i=1}^d\ov{\xi_i}\right)\le 0\right\}     \\
& \le \ex\left[\re\left(\prod_{i=1}^d\ov{\xi_i}\right) (g (\bu_\xi)-v_{\min}(G_m)) \,\Bigg{|}\, \re\left(\prod_{i=1}^d \ov{\xi_i}\right)>0 \right] \Prob\left\{\re\left(\prod_{i=1}^d\ov{\xi_i}\right)>0\right\} \\
& \le \frac{2}{3}\, \ex\left[ g (\bu_\xi)-v_{\min}(G_m) \,\Bigg{|}\, \re\left(\prod_{i=1}^d\ov{\xi_i}\right)>0 \right].
\end{align*}
By randomization of $\xi_i$'s satisfying $\re\left(\prod_{i=1}^d\ov{\xi_i}\right)>0$, we can find $\bu_\beta$ whose components belong to $\conv(\Om_m)$, such that
\begin{align*}
  g(\bu_\beta)-v_{\min}(G_m)  & \ge \ex\left[ g (\bu_\xi)-v_{\min}(G_m) \,\Bigg{|}\, \re\left(\prod_{i=1}^d\ov{\xi_i}\right)>0 \right] \\
  & \ge\frac{3}{2}\,\ex\left[\re\left(\prod_{i=1}^d\ov{\xi_i}\right) g (\bu_\xi) \right] \\
  & \ge\frac{3}{2}\cdot \frac{c_4(m)d!}{(2d)^d}\left(\frac{\delta\ln (2n)}{2n}\right)^{\frac{d-2}{2}}v_{\max}(G_m) \\
  & \ge \frac{c_4(m)d!}{(2d)^d}\left(\frac{\delta\ln (2n)}{2n}\right)^{\frac{d-2}{2}}\left( v_{\max}(G_m) - v_{\min}(G_m) \right),
\end{align*}
where the last two inequalities are due to~\eqref{eq:linking3} and~\eqref{eq:gcase1}, respectively.

In the second case when~\eqref{eq:gcase1} does not hold, we have $v_{\max}(G_{m}) < \frac{2}{3}(v_{\max}(G_{m})-v_{\min}(G_{m}))$, which implies that $-v_{\min}(G_{m})>\frac{1}{3}(v_{\max}(G_{m})-v_{\min}(G_{m}))$. As $g(\boldsymbol{0})=0$, we get
$$
g(\boldsymbol{0})-v_{\min}(G_{m}) > \frac{1}{3}(v_{\max}(G_{m})-v_{\min}(G_{m}))
\ge \frac{c_4(m)d!}{(2d)^d}\left(\frac{\delta\ln (2n)}{2n}\right)^{\frac{d-2}{2}}\left( v_{\max}(G_m) - v_{\min}(G_m) \right).
$$

Combining these two cases, by letting $\bu_\zeta=\arg\max\{g(\boldsymbol{0}),g(\bu_\beta)\}$, we uniformly have
$$
g(\bu_\zeta)-v_{\min}(G_{m})\ge \frac{c_4(m)d!}{(2d)^d}\left(\frac{\delta\ln (2n)}{2n}\right)^{\frac{d-2}{2}}\left( v_{\max}(G_m) - v_{\min}(G_m) \right).
$$
Finally, as $g(\bx)$ is square-free and the components of $\bu_\zeta$ belong to $\conv(\Om_m)$, by Proposition~\ref{thm:convfeas}, we can find $\bz\in\Om_m^n$ such that $g(\bz)\ge g(\bu_\zeta)$, proving the approximation guarantee for the first part.
\end{proof}

\subsection{Conjugate form with the spherical constraint}\label{sec:gform3}

Our final complex polynomial optimization model is $(G_S): \max_{\bx \in \BS^n}g(\bx)$, the maximization of a real-valued general conjugate form with the complex spherical constraint. The problem is also called the largest eigenvalue/eigenvector problem of a conjugate super-symmetric tensor~\cite{JLZ16}. Once again, we provide polynomial-time randomized approximation algorithms with guaranteed worst-case performance ratios. Instead of a discussable flavor presented in Section~\ref{sec:gform1}, here we propose a whole theorem with a complete picture of the proof.

\begin{theorem}
(i) If $d$ is even, then $(G_{S})$ admits a polynomial-time randomized algorithm with relative approximation ratio $\frac{d!}{(2d)^d} \left(\frac{\gamma\ln (2n)}{2n}\right)^{\frac{d-2}{2}}$, i.e., for any given $\gamma \in (0,\frac{2n}{\ln (2n)})$, a feasible solution $\bz\in\BS^n$ can be generated in polynomial time, such that
$$g(\bz)-v_{\min}(G_S)\ge \frac{d!}{(2d)^d} \left(\frac{\gamma\ln (2n)}{2n}\right)^{\frac{d-2}{2}} \left( v_{\max}(G_S) - v_{\min}(G_S) \right).$$
(ii) If $d$ is odd, then $(G_{S})$ admits a polynomial-time randomized algorithm with approximation ratio $\frac{d!}{(\sqrt{2}d)^d} \left(\frac{\gamma\ln (2n)}{2n}\right)^{\frac{d-2}{2}}$.
\end{theorem}
\begin{proof}
When $d=2$, $(G_{S})$ is to find the largest eigenvalue/engenvector of a conjugate super-symmetric tensor $\G$, which is solvable in polynomial time. Therefore in the following proof we assume that $d\ge3$.

When $d$ is even, we first choose any feasible solution $\by\in\BS^n$ and discuss $(G_{S})$ in two cases depending on $g(\by)$. In the first case we assume that
\begin{equation} \label{eq:ball2case}
  g(\by)-v_{\min}(G_{S})\le \frac{2^{-\frac{d}{2}}\tau}{6}(v_{\max}(G_{S})-v_{\min}(G_{S})),
\end{equation}
where $\tau:=\left(\frac{\gamma\ln (2n)}{2n}\right)^{\frac{d-2}{2}}\le1$. Define $h(\bx)=(\ov{\bx}^{\T}\bx)^{\frac{d}{2}}=\|\bx\|_2^d$, a real-valued general conjugate form associated with a super-symmetric conjugate tensor $\HI\in\C^{(2n)^d}$. Consider the following complex multilinear form optimization model
$$
\begin{array}{lll}
(LH_S) & \max  & \re \left( G(\bx^1,\bx^2,\dots,\bx^d) - g(\by) H(\bx^1,\bx^2,\dots,\bx^d) \right) \\
           & \st & \bx^k \in \BS^{2n} ,\, k=1,2,\dots,d,
\end{array}
$$
Applying Theorem~\ref{theo:mulsph}, we can obtain $\bz^1,\bz^2,\dots,\bz^d\in\BS^{2n}$ in polynomial-time, such that
$$
\re \left( G(\bz^1,\bz^2,\dots,\bz^d) - g(\by) H(\bz^1,\bz^2,\dots,\bz^d) \right) \ge \tau v_{\max}(LH_S).
$$
Let $\bx_*\in\BS^n$ be an optimal solution of $(G_S)$. Noticing that $\left(\frac{1}{\sqrt{2}}\binom{\ov{\bx_*}}{\bx_*}, \frac{1}{\sqrt{2}}\binom{\ov{\bx_*}}{\bx_*}, \dots, \frac{1}{\sqrt{2}}\binom{\ov{\bx_*}}{\bx_*} \right)$ is a feasible solution of $(LH_S)$
\begin{align*}
  &~~~~v_{\max}(LH_S) \\
  & \ge G\left(\frac{1}{\sqrt{2}}\binom{\ov{\bx_*}}{\bx_*}, \frac{1}{\sqrt{2}}\binom{\ov{\bx_*}}{\bx_*}, \dots, \frac{1}{\sqrt{2}}\binom{\ov{\bx_*}}{\bx_*} \right) - g(\by) H \left(\frac{1}{\sqrt{2}}\binom{\ov{\bx_*}}{\bx_*}, \frac{1}{\sqrt{2}}\binom{\ov{\bx_*}}{\bx_*}, \dots, \frac{1}{\sqrt{2}}\binom{\ov{\bx_*}}{\bx_*} \right)\\
  &= 2^{-\frac{d}{2}} g(\bx_*) - 2^{-\frac{d}{2}}g(\by) h(\bx_*) \\
  & =  2^{-\frac{d}{2}} (v_{\max}(G_S) - g(\by))
\end{align*}
By the definition of $h(\bx)$, it is easy to see that $|H(\bz^1,\bz^2,\dots,\bz^d)| \le 1$, and we have
\begin{align*}
&~~~~\re\left( G(\bz^1,\bz^2,\dots,\bz^d)-v_{\min}(G_{S})H(\bz^1,\bz^2,\dots,\bz^d)\right)\\
&=\re\left( G(\bz^1,\bz^2,\dots ,\bz^d)-g(\by)H(\bz^1,\bz^2,\dots ,\bz^d)\right) +(g(\by)-v_{\min}(G_{S})) \re\left(H(\bz^1,\bz^2,\dots ,\bz^d)\right)\\
&\ge\tau v_{\max}(LH_S)-(g(\by)-v_{\min}(G_{S}))\\
&\ge 2^{-\frac{d}{2}} \tau (v_{\max}(G_S) - g(\by)) - (g(\by)-v_{\min}(G_{S}))\\
&\ge 2^{-\frac{d}{2}} \tau \left(1- \frac{2^{-\frac{d}{2}} \tau}{6}\right) \left(v_{\max}(G_{S})-v_{\min}(G_{S})\right) - \frac{2^{-\frac{d}{2}}\tau}{6}(v_{\max}(G_{S})-v_{\min}(G_{S}))\\
&\ge\frac{2}{3}\cdot 2^{-\frac{d}{2}}\tau(v_{\max}(G_{S})-v_{\min}(G_{S})),
\end{align*}
where the second last inequality is due to~\eqref{eq:ball2case}. Choose any integer $m\ge3$ or $m=\infty$ and let $\xi_1,\xi_2,\dots,\xi_d$ be i.i.d.\ uniformly on $\Om_m$. Denote $\bz^k=\binom{\bx^k}{\by^k}$ for $k=1,2,\dots,d$ and define
\begin{equation} \label{eq:soldef4}
  \bv_\xi:=\sum_{k=1}^d \left(\ov{\xi_k\bx^k}+\xi_k\by^k\right).
\end{equation}
For any $\bv_\xi$, as $\frac{\bv_\xi}{\|\bv_{\xi}\|_2}$ is a feasible solution of $(G_S)$, $g\left(\frac{\bv_\xi}{\| \bv_{\xi}\|_2} \right) - v_{\min}(G_S) \ge0$, implying that
$$g ( \bv_\xi  )- v_{\min}(G_S) \|\bv_{\xi}\|_2^d \ge 0.$$
By applying the polarization identity in Theorem~\ref{thm:link} to the real-valued general conjugate form $g(\bx)-v_{\min}(G_S)h(\bx)$ and taking the real part, we have
\begin{align*}
&~~~~d!\, \re\left( G(\bz^1,\bz^2,\dots,\bz^d)-v_{\min}(G_{S})H(\bz^1,\bz^2,\dots,\bz^d)\right) \\
& =\ex\left[\re\left(\prod_{i=1}^d\ov{\xi_i}\right) \left(g (\bv_\xi)-v_{\min}(G_S) h(\bv_\xi)\right) \right]\\
& = \ex\left[\re\left(\prod_{i=1}^d\ov{\xi_i}\right) \left(g (\bv_\xi)-v_{\min}(G_S) \|\bv_\xi\|_2^d\right) \,\Bigg{|}\, \re\left(\prod_{i=1}^d\ov{\xi_i}\right)>0 \right] \Prob\left\{\re\left(\prod_{i=1}^d\ov{\xi_i}\right)>0\right\}\ \\
 & ~~~~ + \ex\left[\re\left(\prod_{i=1}^d\ov{\xi_i}\right) \left(g (\bv_\xi)-v_{\min}(G_S) \|\bv_\xi\|_2^d\right) \,\Bigg{|}\, \re\left(\prod_{i=1}^d\ov{\xi_i}\right)\le 0 \right] \Prob\left\{\re\left(\prod_{i=1}^d\ov{\xi_i}\right)\le 0\right\}     \\
& \le \ex\left[\re\left(\prod_{i=1}^d\ov{\xi_i}\right) \left(g (\bv_\xi)-v_{\min}(G_S) \|\bv_\xi\|_2^d\right) \,\Bigg{|}\, \re\left(\prod_{i=1}^d \ov{\xi_i}\right)>0 \right] \Prob\left\{\re\left(\prod_{i=1}^d\ov{\xi_i}\right)>0\right\} \\
& \le \frac{2}{3}\, \ex\left[ g (\bv_\xi)-v_{\min}(G_S) \|\bv_\xi\|_2^d \,\Bigg{|}\, \re\left(\prod_{i=1}^d\ov{\xi_i}\right)>0 \right],
\end{align*}
where the last inequality is due to~\eqref{eq:xipositive}. By randomization of $\xi_i$'s satisfying $\re\left(\prod_{i=1}^d\ov{\xi_i}\right)>0$, we can find $\bv_\beta$, such that
\begin{align*}
  g (\bv_\beta)-v_{\min}(G_S) \|\bv_\beta\|_2^d &\ge   \ex\left[ g (\bv_\xi)-v_{\min}(G_S) \|\bv_\xi\|_2^d \,\Bigg{|}\, \re\left(\prod_{i=1}^d\ov{\xi_i}\right)>0 \right] \\
  &\ge \frac{3 d!}{2}\re\left( G(\bz^1,\bz^2,\dots,\bz^d)-v_{\min}(G_{S})H(\bz^1,\bz^2,\dots,\bz^d)\right) \\
  &\ge  d!\, 2^{-\frac{d}{2}}\tau(v_{\max}(G_{S})-v_{\min}(G_{S})).
\end{align*}
By noticing that $\|\bv_\beta\|_2 \le \sum_{k=1}^d \left(\|\bx^k\|_2+\|\by^k\|_2\right) \le \sum_{k=1}^d \sqrt{2}\|\bz^k\|_2=\sqrt{2}d$, we get
\begin{align*}
  g \left(\frac{\bv_\beta}{\|\bv_\beta\|_2}\right)-v_{\min}(G_S) &= \frac{1}{\|\bv_\beta\|_2^d} \left( g (\bv_\beta)-v_{\min}(G_S) \|\bv_\beta\|_2^d \right) \\
   &\ge (\sqrt{2}d)^{-d} d!\, 2^{-\frac{d}{2}} \tau (v_{\max}(G_{S})-v_{\min}(G_{S})) \\
   &= \frac{d!}{(2d)^d} \left(\frac{\gamma\ln (2n)}{2n}\right)^{\frac{d-2}{2}}(v_{\max}(G_{S})-v_{\min}(G_{S})).
\end{align*}

In the second case when~\eqref{eq:ball2case} does not hold, as $d\ge3$, we have
$$g(\by)-v_{\min}(G_{S})> \frac{2^{-\frac{d}{2}}\tau}{6}(v_{\max}(G_{S})-v_{\min}(G_{S})) \ge \frac{d!}{(2d)^d} \left(\frac{\gamma\ln (2n)}{2n}\right)^{\frac{d-2}{2}}(v_{\max}(G_{S})-v_{\min}(G_{S})).$$
Finally, by combining these two cases and letting $\bz=\arg\max\left\{g(\by),g \left(\frac{\bv_\beta}{\|\bv_\beta\|_2}\right)\right\}\in\BS^n$, the relative approximation ratio of $g(\bz)$ is guaranteed.

When $d$ is odd, applying the tensor relaxation method, $(G_S)$ can be relaxed to
$$
\begin{array}{lll}
(LG_S) & \max  & \re G(\bx^1,\bx^2,\dots,\bx^d) \\
           & \st & \bx^k \in \BS^{2n} ,\, k=1,2,\dots,d.
\end{array}
$$
By Theorem~\ref{theo:mulsph}, $\bz^1,\bz^2,\dots ,\bz^d\in\BS^{2n}$ can be generated in polynomial time such that
$$\re G(\bz^1,\bz^2,\dots,\bz^d) \ge \left(\frac{\gamma\ln (2n)}{2n}\right)^{\frac{d-2}{2}} v_{\max}(LG_S) \ge \left(\frac{\gamma\ln (2n)}{2n}\right)^{\frac{d-2}{2}} v_{\max}(G_S).$$
Choose any integer $m\ge3$ or $m=\infty$ and let $\xi_1,\xi_2,\dots,\xi_d$ be i.i.d.\ uniformly on $\Om_m$. Denote $\bz^k=\binom{\bx^k}{\by^k}$ for $k=1,2,\dots,d$ and define $\bv_\xi$ as that in~\eqref{eq:soldef4}. By the polarization identity in Theorem~\ref{thm:link}
$$
\ex\left[\re\left(\prod_{i=1}^d\ov{\xi_i}\right) g (\bv_\xi) \right]\\
\ge d!\, \re G(\bz^1,\bz^2,\dots,\bz^d) \ge d! \left(\frac{\gamma\ln (2n)}{2n}\right)^{\frac{d-2}{2}} v_{\max}(G_S).
$$
By randomization, we may find $\alpha_1,\alpha_2,\dots,\alpha_d\in\Om_m$ such that
$$
\re\left(\prod_{i=1}^d\ov{\alpha_i}\right) g (\bv_\alpha)
\ge \ex\left[\re\left(\prod_{i=1}^d\ov{\xi_i}\right) g (\bv_\xi) \right]
\ge d! \left(\frac{\gamma\ln (2n)}{2n}\right)^{\frac{d-2}{2}} v_{\max}(G_S).
$$
Noticing that $\|\bv_\alpha\|_2 \le \sum_{k=1}^d \left(\|\bx^k\|_2+\|\by^k\|_2\right) \le \sum_{k=1}^d \sqrt{2}\|\bz^k\|_2=\sqrt{2}d$ and $g(-\bx)=-g(\bx)$ as $d$ is odd, and letting
$\bz=\arg\max\left\{g \left(\frac{\bv_\alpha}{\|\bv_\alpha\|_2}\right),g \left(-\frac{\bv_\alpha}{\|\bv_\alpha\|_2}\right)\right\}\in\BS^n$, we finally get
$$
g(\bz) = \left| g \left(\frac{\bv_\alpha}{\|\bv_\alpha\|_2}\right) \right| = \frac{|g(\bv_\alpha)|}{\|\bv_\alpha\|_2^d}
\ge \frac{1}{(\sqrt{2}d)^d} \re\left(\prod_{i=1}^d\ov{\alpha_i}\right) g (\bv_\alpha)
\ge \frac{d!}{(\sqrt{2}d)^d} \left(\frac{\gamma\ln (2n)}{2n}\right)^{\frac{d-2}{2}} v_{\max}(G_S).
$$
\end{proof}

%
%
%
%
%

\end{document}